\documentclass[12pt]{article}

\usepackage{amsmath}
\usepackage{amssymb}
\usepackage{amsthm}

\title{Generic stability, regularity, and quasiminimality}
\date{September 19th, 2010 }
\author{Anand Pillay\thanks{Supported
by EPSRC grant EP/F009712/1}\\University of Leeds \and Predrag
Tanovi\'c\thanks{Supported by the Ministry of Science and
Technology of Serbia}\\Mathematical Institute SANU, Belgrade}

\newcommand{\nequiv}{\makebox[1em]{$\not$\makebox[.6em]{$\equiv$}}}
\newcommand \nmodels {\mathop{\not\models}}
\def \tp {{\rm tp}}

\def \acl {{\rm acl}}

\def \ecl {{\rm ecl}}
\def \Th {{\rm Th}}

\def \cl {{\rm cl}}
\def \Cl {{\rm Cl}}
\def \ccl {{\rm ccl}}
\def \Aut {{\rm Aut}}

\def \lqq {\,\preccurlyeq\,}

\def \nlqq {\,\mathop{\not \!  \preccurlyeq }\,}

\newtheorem{thm}{Theorem}
 \newtheorem{prop}{Proposition}[section]
 \newtheorem{lem}{Lemma}[section]
 \newtheorem{cor}{Corollary}

\theoremstyle{definition}
 \newtheorem{exm}{Example}[section]
 \newtheorem{defn}{Definition}[section]

\theoremstyle{remark}
 \newtheorem{rmk}{Remark}[section]
 \newtheorem*{Qst}{Question}

\begin{document}

\maketitle

\begin{abstract}
We study the notions generic stability, regularity, homogeneous
pregeometries, quasiminimality, and their mutual relations, in
arbitrary first order theories. We prove that
``infinite-dimensional homogeneous pregeometries" coincide with
generically stable strongly regular types $(p(x), x=x)$. We prove
that quasiminimal structures of cardinality at least $\aleph_{2}$
are ``homogeneous pregeometries". We prove that the ``generic
type" of an arbitrary quasiminimal structure is ``locally strongly
regular". Some of the results depend on a general dichotomy for
``regular-like" types: generic stability, or the existence of a
suitable definable partial ordering.
\end{abstract}

\maketitle

\section{Introduction}\label{s1}
The first author was motivated partly by hearing Wilkie's talks on his program for proving Zilber's
conjecture that the complex exponential field is quasiminimal (definable subsets are
countable or co-countable), and wondering about the first order (rather than infinitary)
consequences of the approach. The second author was partly motivated by his
interest in adapting his study of minimal structures (definable subsets are finite or cofinite)
and his dichotomy theorems (\cite{Tan}), to
the quasiminimal context.

Zilber's conjecture (and the approach to it outlined by Wilkie) is
closely related to the existence and properties of a canonical
pregeometry on the complex exponential field. See the end of
section \ref{s4} for a more detailed discussion. Also in section
\ref{s5} of this paper we discuss to what extent a pregeometry can
be recovered just from quasiminimality, sometimes assuming the
presence of a definable group structure. This continues in a sense
an earlier study of the general model theory of quasiminimality by
Itai, Tsuboi, and Wakai \cite{ITW}.

A pregeometry is a closure relation on subsets of a not
necessarily saturated structure $M$ satisfying usual properties
(including exchange). We will also assume ``homogeneity"
($\tp(b/A)$ is unique for $b\notin \cl(A)$) and
``infinite-dimensionality" ($\dim(M)$ is infinite). One of the
points of this paper  then is that the canonical ``generic type"
$p$ of the pregeometry is ``generically stable" and regular. This
includes the statement that on realizations of $p$, the closure
operation is precisely forking in the sense of Shelah (see Theorem
\ref{Tpr}). Another main point of the paper is that a quasiminimal
structure of cardinality at least $\aleph_{2}$  carries, in a
canonical fashion, a homogeneous (in the above sense),
infinite-dimensional pregeometry, improving on results in
\cite{ITW}.

Generic stability, the stable-like behavior of a given complete
type vis-a-vis forking, was studied in several papers including
\cite{Shelah783} and \cite{NIPII}, but mainly in the context of
theories with $NIP$  (i.e. without the independence property).
Here we take the opportunity, in section \ref{s2},  to give
appropriate definitions in an arbitrary ambient theory $T$, as
well as discussing generically stable (strongly) regular types.

The notion of a regular type is central in stability theory and
classification theory, where the counting of models of superstable
theories is related to dimensions of regular types. Here (section
\ref{s3}) we give appropriate generalizations of (strong)
regularity for an arbitrary theory $T$ (although it does not agree
with the established definitions for simple theories). In section
\ref{s3} a basic dichotomy theorem (Theorem \ref{Pr1}) is proved
for global regular types $p$; roughly speaking, either $p$ is
generically stable, or there is  certain definable partial
ordering on the set of realizations of $p$. A generalization of
this theorem is given in section \ref{s6} (Theorem \ref{Tqg}),
which lies behind our result on quasiminimal structures of
cardinality $\geq \aleph_{2}$.

In section \ref{s7} a local version of regularity is given and
applied to the analysis of quasiminimal structures (see
Corollaries \ref{C3} and \ref{C4}).

\vspace{2mm}   The current paper is a revised and expanded version
of a preprint ``Remarks on generic stability, pregeometries, and
quasiminimality" by the first author, which was written and
circulated in June 2009.

\vspace{2mm}  The first author would like to thank Clifton Ealy,
Krzysztof Krupinski, and Alex Usvyatsov for various helpful
conversations and comments. After seeing the first author's
preliminary  results (and talk at a meeting in Lyon in July 2009),
Ealy pointed out that the commutativity of regular groups is
problematic (see our Question at the end of section \ref{s3}) and
suggested possible directions towards a counterexample. Krupinski
independently came up with examples such as Example \ref{field} of
the current paper. And after a talk by the first author on the
same subject in Bedlewo in August 2009, Usvyatsov pointed out
examples such as Example \ref{E1} and \ref{E2} of the current
paper. Both authors would like to thank Jonathan Kirby for helpful
comments and questions, and for clarifying the connection to
exponential algebraicity (see the end of section \ref{s4}).

\vspace{2mm} We now give our conventions and give a few basic
definitions relevant to the paper. As the referee kindly mentioned
this paper maybe of interest to readers who are not so familiar
with stability-style model theory, and so following his/her
suggestion, we give more details than we usually would concerning
some standard constructions.

$T$ denotes an arbitrary complete  $1$-sorted theory in a language
$L$ and $\bar{M}$ denotes a saturated (monster) model of $T$. As a
rule $a,b,c,...$ denote elements of $\bar M$, and
$\bar{a},\bar{b},\bar{c}$  denote finite tuples of elements. (But
in some situations $a,b,..$ may denote elements of ${\bar
M}^{eq}$.) $A,B,C$ denote  small subsets, and  $M,M_0, ...$ denote
small elementary submodels of $\bar{M}$. By a ``global type" we mean morally a complete type
over a sufficiently saturated model. In practise we will mean a complete type
$p(\bar{x})\in S(\bar{M})$ over the ``monster model". Such a type $p$ is said to be $A$-invariant if
 $p$ is
$\Aut(\bar{M}/A)$-invariant;  and $p$ is said to be invariant   if
it is $A$-invariant for some small $A$. Notice that by
saturation/homogeneity of  ${\bar M}$, the $A$-invariance of $p$
is equivalent to $p$ {\em not splitting} over $A$, in the sense
that for any $L$-formula $\phi(\bar{x},\bar{y})$ and $\bar{b}$
from ${\bar M}$, whether or not $\phi(\bar{x},\bar{b})\in p$
depends on $\tp(\bar{b}/A)$. So an $A$-invariant type $p(\bar{x})$
comes with a kind of ``infinitary" defining schema $d_{p}$ over
$A$. Namely for a given $L$-formula  $\phi(\bar{x},\bar{y})$,
$d_{p}(\phi(\bar{x},\bar{y}))$ is the set of  $q(\bar{y})\in S(A)$
such that for some (any) $\bar{b}$ realizing $q$,
$\phi(\bar{x},\bar{b}) \in p$.

We now explain how to build  the ``nonforking iterates"  $p^{(n)}(\bar{x}_{1},..,\bar{x}_{n})\in S({\bar M})$ of $p$.
Let $\phi(\bar{x}_{1}, \bar{x}_{2},\bar{c})$ be a formula over ${\bar M}$  (with witnessed parameters $\bar{c}$).
We will put such a formula  in $p^{(2)}(\bar{x}_{1},\bar{x}_{2})$ if  for some (any) $\bar{a}_{1}$ realizing $p|(A,\bar{c})$  (the restriction of $p$ to $A,\bar{c}$), the formula  $\phi(\bar{a}_{1},\bar{x}_{2},\bar{c})$ is in $p(\bar{x}_{2})$. Having defined  $p^{(n)}$, we put a formula
$\phi(\bar{x}_{1},..,\bar{x}_{n},\bar{x}_{n+1},\bar{c})$ in $p^{(n+1}$ if  for for some (any)
$(\bar{a}_{1},...,\bar{a}_{n})$ realizing $p^{(n)}|(A,\bar{c})$, the formula
$\phi(\bar{a}_{1},..,\bar{a}_{n},\bar{x}_{n+1},\bar{c})$ is in $p(\bar{x}_{n+1})$.   This construction depends only on the defining schema $d$ of $p$.  It is clear that
$p^{(n)}(\bar{x}_{1},..,\bar{x}_{n}) \subseteq  p^{(n+1)}(\bar{x}_{1},...,\bar{x}_{n+1})$, and we let $p^{(\omega)}$ be the increasing union of the $p^{(n)}$'s. It is not hard to see that any realization $(\bar{a}_{i}:i=1,2,....)$ of $p^{(\omega)}$  (in an elementary extension of the monster model!) is an indiscernible sequence over ${\bar M}$.

Assuming the global type $p$ to be $A$-invariant, by a {\em Morley sequence in $p$ over $A$} we mean a realization, in ${\bar M}$, $(\bar{a}_{i}:i=1,2,....)$ of $p^{(\omega)}|A$.  Clearly this can also be obtained by choosing $\bar{a}_{1}$ to realize $p|A$, choosing  $\bar{a}_{2}$ to realize  $p|(A,\bar{a}_{1})$ etc. So a Morley sequence in $p$ over $A$ is among other things an $A$-indiscernible sequence, and can be stretched to an $A$-indiscernible sequence of any ordinal length  $\alpha$
(a Morley sequence in $p$ over $A$ of length $\alpha$). In any case clearly the type over $A$ of any Morley sequence of length $\alpha$ in $p$ over $A$ depends only on $p$ and $A$.

An invariant global type $p(\bar{x})$ is said to be {\em symmetric} if for any $n$, formula
$\phi(\bar{x}_{1},..,\bar{x}_{n})$ over ${\bar M}$ and permutation $\pi$ of $\{1,..,n\}$,
$\phi(\bar{x}_{1},...,\bar{x}_{n}) \in p^{(n)}(\bar{x}_{1},..,\bar{x}_{n})$ if and only if
$\phi(\bar{x}_{\pi(1)},...,\bar{x}_{\pi(n)}) \in p^{(n)}(\bar{x}_{1},..,\bar{x}_{n})$.

It is not hard to see that the condition for $n=2$ implies it for all $n$. Also note that $p$ is symmetric if and only if for any small $A$ such that $p$ is $A$-invariant, any Morley sequence in $p$ over $A$ is {\em totally indiscernible}.

In some parts of this paper we will be considering  analogous notions for complete types over ``large" but not necessarily saturated or homogeneous (in the sense of model theory) structures $M$, where we are no longer able to use the expression ``invariant" type.

Some other basic notions used in this paper are {\em definable type}, {\em heir}, {\em coheir}, {\em almost finitely satisfiability}. A  {\em definable type} (over $A$) refers usually to a complete type $p(x)$ say over a model $M$ such that for any $\phi(x,\bar{y})\in L$, $\{\bar{b}\in M:\phi(x,\bar{b})\in p\}$ is a definable set (over $A$) in $M$.
If $p(x)$ is a complete type over a model $M$ and $N$ is a larger model (elementary extension) then an {\em heir} of $p$ over $N$ is an extension $q(x)\in S(N)$ of $p$ such that for any formula $\phi(x,\bar{y})$ with parameters from $M$, if $\phi(x,\bar{c})\in q(x)$ for some $\bar{c}$ from $N$, then $\phi(x,\bar{b})\in p$ for some $\bar{b}\in M$. A basic
and easy fact is that if $p(x)\in S(M)$ is a definable type, then it has a unique heir over any larger model (which is precisely given by applying the defining schema for $p$ to the larger model).

In the same context, with $p(x)\in S(M)$, and $q(x)\in S(N)$ an extension of $p$, $q$ is said to be a {\em coheir} of $p$ if $q$ is finitely satisfiable in $M$ (any formula $\phi(x)$ in $q(x)$ is realized by some element of $M$, even though the formula may have parameters outside $M$).   Finally if $q(x)$ is a global complete type and $A$ a small set of parameters, $q$ is said to be {\em almost finitely satisfiable} in $A$ if $q$ is finitely satisfiable in any model (elementary substructure) of ${\bar M}$ which contains $A$.

Of course pervasive notions are {\em dividing} and {\em forking}
in the sense of Shelah. A formula $\phi(x,b)$ (with witnessed
parameters $b$) is said to {\em divide over $A$} if for some
$A$-indiscernible sequence $(b_{i}:i<\omega)$ of realizations of
$\tp(b/A)$, $\{\phi(x,b_{i}:i<\omega\}$ is inconsistent. The
formula is said to {\em fork over $A$} if it implies a finite
disjunction of formulas $\psi_{j}(x,c_{j})$ each of which divides
over $A$. The notions extend naturally to  (partial, or complete)
types in place of formulas.

In stable theories, the above notions cohere in the following
senses:
\begin{enumerate}
 \item Given a type $p(x)\in S(M)$ over a model and $N
\supseteq M$, $p$ has a unique nonforking extension over $N$ which
coincides with its unique heir over $N$ as well as with its unique
coheir over $N$.

\item Any complete type over a model is definable (in particular
any global type is invariant).

\item A global type $p(x)$ does not fork over $A$ iff it is
definable over $\acl^{eq}(A)$ iff it is almost finitely
satisfiable in $A$.

\item Moreover any global type is definable in a specific way over
some (any) Morley sequence. (See Proposition \ref{GS}(i) below.)
\end{enumerate}

\noindent
Although stable theories were originally the main environment for studying forking,
we saw subsequently the fundamental role of forking in simple theories,
and more recently its importance for $NIP$ theories, including $o$-minimal
theories and valued fields. In the current paper we are concerned to a large extent
with the role of forking in ``pregeometries" in general first order theories.

\medskip
Recall that $(P,\cl)$, where $\cl$ is an operation on subsets of
$P$, is a \emph{pregeometry} (or $\cl$ is a pregeometry on $P$) if
for all $A,B\subseteq P$ and $a,b\in P$ the following hold:

\begin{description}
\item[\rm Monotonicity]    \ \   $A\subseteq B$ \ implies \
$A\subseteq\cl(A)\subseteq\cl(B)$;

\item[\rm Finite character]  \ \ $\cl(A)=\bigcup
\{\cl(A_0)\,|\,A_0\subseteq A$ finite$\}$;

\item[\rm Transitivity]       \  \ $\cl(A)=\cl(\cl(A))$;

\item[\rm Exchange (symmetry)] \ \
$a\in\cl(A\cup\{b\})\setminus\cl(A)$ \  implies  \
$b\in\cl(A\cup\{a\})$.   \end{description} $\cl$ is called a
\emph{closure operator}    if it satisfies the first three
conditions.

\section{Generic stability}\label{s2}

\begin{defn}\label{D1}
A non-algebraic global type $p(\bar{x})\in S(\bar{M})$ is
generically stable if, for some small $A$,  it is $A$-invariant
and:
\begin{center}
if $(\bar{a}_i\,:\,i<\alpha)$ (any $\alpha$, not only $\omega$) is a Morley sequence in $p$ over $A$
then for any formula $\phi(\bar{x})$ (with parameters from
$\bar{M}$) \ $\{i\,:\ \models\phi(\bar{a}_i)\}$ is either finite
or co-finite.\end{center}
\end{defn}

\begin{rmk}
If $p$ is generically stable then as a witness-set $A$ in the
definition we can take any small $A$ such that $P$ is
$A$-invariant.
\end{rmk}

\begin{prop}\label{GS}
Let $p(\bar{x})\in S(\bar{M})$ be generically stable and
$A$-invariant. Then:

\begin{enumerate}
\item[(i)]   For any formula $\phi(\bar{x},\bar{y})\in L$ there is
$n_{\phi}$ such that for any Morley sequence
$(\bar{a}_i\,:\,i<\omega)$ of $p$ over $A$, and   any $\bar{b}$:
\begin{center}
 $\phi(\bar{x},\bar{b})\in p $ \ \ \   iff   \ \ \
$\models\bigvee_{w\subset\{0,1,...2n_{\phi}\},\,|w|=n_{\phi}+1}\bigwedge_{i\in
w}\phi(\bar{a}_i,\bar{b}).$\end{center}

\item[(ii)]   $p$ is definable over $A$ and almost finitely
satisfiable in $A$.

\item[(iii)]   Any Morley sequence of $p$ over $A$ is totally
indiscernible.

\smallskip
\item[(iv)]   $p$ is the unique global nonforking extension of
$p\,|\,A$.
\end{enumerate}
\end{prop}
\begin{proof} This is a slight elaboration of the proof of
Proposition 3.2 from \cite{NIPII}.

\smallskip
(i) \ First note that for any $\phi(\bar{x},\bar{y})\in L$ there
is $n_{\phi}$ such that for any Morley sequence
$(\bar{a}_i\,:\,i<\omega)$ in $p$ over $A$, and any $\bar{b}$
either at most $n_{\phi}$ many $\bar{a}_i$'s satisfy
$\phi(\bar{x},\bar{b})$ or at most $n_{\phi}$ many $\bar{a}_i$'s
satisfy $\neg\phi(\bar{x},\bar{b})$. For if not, then by
compactness we can find a Morley sequence in $p$ over $A$ of
length $\omega + \omega$ which violates Definition \ref{D1}.  Note
that if at most $n_{\phi}$ $\bar{a}_{i}$'s satisfy
$\neg\phi(\bar{x},\bar{b})$ then $\phi(\bar{x},\bar{b})\in p$: for
otherwise we could let $(\bar{a}_{\omega + j}:j<\omega)$ realize
$p^{(\omega)}$ restricted to
$A\cup\{\bar{a}_{i}:i<\omega\}\cup\{\bar{b}\}$, in which case
$(\bar{a}_{i}:i<\omega + \omega)$ is a Morley sequence in $p$ over
$A$ such that all but finitely many $\bar{a}_{i}$ for $i < \omega$
satisfy $\phi(\bar{x},\bar{b})$ and  all  $\bar{a}_{\omega + j}$
satisfy $\neg\phi(\bar{x},\bar{b})$, again a contradiction.

\smallskip
(ii) \ By (i) $p$ is definable (over a Morley sequence), so  by
$A$-invariance $p$ is definable over $A$. (If  $\psi(\bar{y},\bar{d})$ is the
$\phi(\bar{x},\bar{y})$ definition of $p$, then
by $A$-invariance of $p$, this formula is preserved, up to equivalence,
by automorphisms which fix $A$ pointwise, hence is equivalent to a formula over $A$.)
Let $M$ be a model containing $A$,
$\phi(\bar{x},\bar{c})\in p$ and let $I=(\bar{a}_i\,:\,i<\omega)$
be a Morley sequence in $p$ over $A$ such that $\tp(I/M\bar{c})$
is finitely satisfiable in $M$. Then, by (i),
$\phi(\bar{x},\bar{c})$ is satisfied by some $\bar{a}_i$ hence
also by an element of $M$.

\smallskip
(iii) \ This follows from (i) exactly as in the proof of the
Proposition 3.2 of \cite{NIPII}  (where NIP was not used).

\smallskip
(iv) \ Let
$q(\bar{x})$ be a global nonforking extension of $p\,|\,A$. We
will prove that $q=p$.

\smallskip
\noindent  \emph{Claim.}    Suppose
$(\bar{a}_0,....,\bar{a}_n,\bar{b})$ are such that
$\bar{a}_0\models q\,|\,A$,\, $\bar{a}_{i+1}\models
q\,|\,(A,\bar{a}_0,....,\bar{a}_i)$ and $b\models
q\,|\,(A,\bar{a}_0,....,\bar{a}_n)$. Then
$(\bar{a}_0,....,\bar{a}_n,\bar{b}) $ is a Morley sequence in $p$
over $A$.

\smallskip
\noindent\emph{Proof.}  \ We prove it by induction. Suppose we
have chosen $\bar{a}_0,...,\bar{a}_n$ as in the claim and we know
(induction hypothesis) that $(\bar{a}_0,...,\bar{a}_n)$ begins a
Morley sequence $I=(\bar{a}_i\,|\,i<\omega)$ in $p$ over $A$.
Suppose that $\phi(\bar{a}_0,...,\bar{a}_n,\bar{x})\in q$. Then we
claim that $\phi(\bar{a}_0,...,\bar{a}_{n-1},\bar{a}_i,x)\in q$
for all $i>n$. For otherwise, without loss of generality \
$\neg\phi(\bar{a}_0,...,\bar{a}_{n-1},\bar{a}_{n+1},\bar{x})\in
q$. But then, by indiscernibility of
$\{\bar{a}_i\bar{a}_{i+1}\,:\,i=n,n+2,n+4,...\}$ over
$(A,\bar{a}_0,...,\bar{a}_{n-1})$, and the nondividing of $q$ over
$A$, we have that
\begin{center} $\{\phi(\bar{a}_0,...,\bar{a}_{n-1},\bar{a}_i,x)\wedge
\neg\phi(\bar{a}_0,...,\bar{a}_{n-1},\bar{a}_{i+1},\bar{x})\,:\,i=n,n+2,n+4,...\}$
\end{center}
is consistent, which contradicts Definition \ref{D1}. Hence
\,$\models \phi(\bar{a}_0,...,\bar{a}_{n-1},\bar{a}_i,\bar{b})$
for all $i>n$. So by part (i),
$\phi(\bar{a}_0,...,\bar{a}_{n-1},\bar{x},\bar{b})\in p(\bar{x})$.
The inductive assumption gives that
$\tp(\bar{a}_0,...,\bar{a}_{n-1},\bar{a}_n/A)=
\tp(\bar{a}_0,...,\bar{a}_{n-1},\bar{b}/A)$, so by $A$-invariance
of $p$, $\phi(\bar{a}_0,...,\bar{a}_{n-1},\bar{x},\bar{a}_n)\in
p(\bar{x})$. Thus
$\models\phi(\bar{a}_0,...,\bar{a}_{n-1},\bar{a}_{n+1},\bar{a}_n)$
and, by total indiscernibility of $I$ over $A$ (part (iii)),
$\models\phi(\bar{a}_0,...,\bar{a}_{n-1},\bar{a}_n,\bar{a}_{n+1})$.
We have shown that
$q\,|\,(A,\bar{a}_0,...,\bar{a}_n)=p\,|\,(A,\bar{a}_0,...,\bar{a}_n)$,
which allows the induction process to continue. The claim is
proved.

\smallskip
Now suppose that $\phi(\bar{x},\bar{c})\in q(\bar{x})$. Let
$\bar{a}_i$ for $i<\omega$ be such that $\bar{a}_i$ realizes
$q\,|\,(A,\bar{c},\bar{a}_0,...,\bar{a}_{i-1})$ for all $i$. By
the claim $(\bar{a}_i\,;\,i<\omega)$ is a Morley sequence of $p$
over $A$. But $\phi(\bar{a}_i,\bar{c})$ for all $i$, hence by the
proof of (i), $\phi(\bar{x},\bar{c})\in p$. So $q=p$.
\end{proof}

Let us note for the record that for a global type $p$, generically stable implies
definable implies invariant, and these are strict implications.

\smallskip
Now we restate the notion of generic stability for groups from
\cite{NIPII} in the context of connected groups. Recall that a
definable (or even  type-definable) group $G$ is connected if it
has no relatively definable subgroup of finite index. A type
$p(x)\in S_G(\bar{M})$ is left $G$-invariant iff for all $g\in G$:
\begin{center}
\ $(g\cdot p)(x)=^{def}\{\phi(g^{-1}\cdot x)\,: \phi(x)\in
p(x)\}=p(x)$;\end{center} likewise for right $G$-invariant.

\begin{defn}
Let $G$ be a definable (or even   type-definable) connected group
in $\bar{M}$. $G$ is \emph{generically stable} if there is a
global complete 1-type $p(x)$ extending '$x\in G$' such that $p$
is generically stable and left $G$-invariant.
\end{defn}

As in the NIP context, we show that a generically stable left invariant type  is
also right invariant and  is unique such (and we will call it the
generic type):

\begin{lem}
Suppose that $G$ is generically stable, witnessed by $p(x)$. Then
$p(x)$ is the unique left-invariant   and  also the unique
right-invariant type.
\end{lem}
\begin{proof} First we prove that $p$ is right invariant.
Let $(a,b)$ be a Morley sequence in $p$ over (any small)   $A$. By
left invariance \,$g=a^{-1}\cdot b\models p\,|\,A$. By total
indiscernibility we have $\tp(a,b)=\tp(b,a)$, so
\,$g^{-1}=b^{-1}\cdot a\models p\,|\,A$. This proves that
$p=p^{-1}$. Now, for any $g\in G$ we have \
\begin{center}
$p=g^{-1}\cdot p=(g^{-1}\cdot p)^{-1}=p^{-1}\cdot g=p\cdot
g$ , \end{center} so $p$ is  also right-invariant.

For the uniqueness, suppose that $q$ is a left invariant global
type and we prove that $p=q$. Let    $\phi(x)\in q$ be over $A$,
let  $I=(a_i\,:\,i\in\omega)$ be a Morley sequence in $p$ over
$A$, and let $b\models q\,|\,(A,I)$.  Then, by left invariance of
$q$, for all $i\in\omega$ we have $\models\phi(a_i\cdot b)$. By
Proposition \ref{GS}(i) we get $\phi(x\cdot b)\in p(x)$ and, by
right invariance, $\phi(x)\in p(x)$. Thus $p=q$.
\end{proof}

\section{Global regularity}\label{s3}

For stable $T$, a stationary type $p(x)\in S(A)$ is said to be
regular if for any $B\supseteq A$, and $b$ realizing a forking
extension of $p$ over $B$, $p\,|\,B$ (the unique nonforking
extension of $p$ over $B$) has a unique complete extension over
$(B,b)$. Appropriate versions of regularity (where we do not have
stationarity) have been given  for simple theories for example.
Here we  give somewhat different versions for global ``invariant"
types in arbitrary theories. So the reader should be aware that our definitions
below are not consistent with usual terminology for {\em simple theories},
although they are of course consistent with the stable case.

\begin{defn}\label{Dreg}
Let $p(\bar{x})$ be  a global non-algebraic  type.

\begin{enumerate}
\item[(i)]   $p(\bar{x})$ is said to be \emph{regular} if, for
some small $A$, it is $A$-invariant  and for any $B\supseteq A$
and $\bar{a}\models p\,|\,A$: \ either $\bar{a}\models p\,|\,B$ \
or \ $p\,|\,B\vdash p\,|\,B\bar{a}$.

\item[(ii)]    Suppose  $\phi(\bar{x})\in p$. We say that
$(p(\bar{x}),\phi(\bar{x}))$ is \emph{strongly regular} if, for
some small $A$ over which $\phi$ is defined, $p$ is $A$-invariant
and  for all $B\supseteq A$ and $\bar{a}$ satisfying
$\phi(\bar{x})$, either \ $\bar{a}\models p\,|\,B$ \  or \
$p\,|\,B\vdash p\,|\,B\bar{a}$.\end{enumerate}
\end{defn}

\begin{rmk}\label{Rreg}
If   $p$ is  regular then    as a witness-set $A$ in
the definition we can take any small $A$ such that $p$ is
$A$-invariant. Similarly for strongly regular  types. Also note that strongly
regular implies regular.
\end{rmk}

\begin{defn}\label{D4} Let  $N$ be any submodel  of $\bar{M}$ (possibly $N=\bar{M}$),
and let  $p(x)\in S_1(N)$. The operator $\cl_p$ is defined on
(all) subsets of $N$ by:   \begin{center} $\cl_p(X)=\{a\in N\,|\,
a\nmodels p\,|\,X\}$.\end{center} Also define \
$\cl^A_p(X)=\cl_p(X\cup A)$ \ for any $A\subset N$ (and
$X\subseteq N$).
\end{defn}

Intuitively, having fixed $p(x)\in S(N)$ and $B\subset N$ then
realizations of $p\,|B$ are considered as `generic over $B$' so
formulas in $p(x)$ should be considered as defining `large'
subsets of $N$ and their negations as defining `small' subsets. In
this way, $\cl_p(B)$ is the union of all `small' definable
subsets. In the rest of this section we will be interested in the
case where $N = \bar M$. But later in the paper we will consider
other $N$.

\begin{rmk}\label{R22} Let $p(x)\in S_1(\bar{M})$  be    $A$-invariant,
where $A$ is small.

\smallskip
(i) \ Strong   regularity of $(p(x),x=x)$   translates into a
simpler  expression  using $\cl_p$\,:
\begin{center}
 $(p(x),x=x)$ is strongly regular   \  iff \
$p\,|\,B\vdash p\,|\,\cl_p(B)$ \ for any small $B\supseteq
A$.\end{center}Since $a\models p\,|\,B$ is the same as
$a\notin\cl_p(B)$, another equivalent way of expressing strong
regularity of $(p(x),x=x)$ is:
\begin{center}
$a\models p|B$ \  iff \ $a\models p\,|\,\cl_p(B)$ \ \ (for all
$a\in\bar{M}$ and small $B\supseteq A$).\end{center}

\smallskip
(ii) \ The corresponding expression is not so concise when $(p(x),\phi(x))$ is
strongly regular or when $p$ is just regular. For
example:
\begin{center} $p$ is
regular \  iff \  $p\,|\,AB\vdash p\,|\,A\cup
B\cup(\cl^A_p(B)\cap (p|A)(\bar{M}))$  for any small
$B$.\end{center}
We can also consider the restriction of $\cl^A_p$ to
$A\cup (p|A)(\bar{M})$. Formally we define (for $B$ a small subset of
$A\cup (p|A)(\bar{M})$)
\ $\cl_{p_A}(B)= A\cup (\cl^A_p(B)\cap
(p|A)(\bar{M}))$. Then we have a simple consequence of regularity;
$p$ is regular \ {\em implies} \
\begin{center}
$p\,|\,AB\vdash p\,|\,\cl_{p_A}(B)$ \ for any small $B\subset
(p|A)(\bar{M})$.
\end{center}
As in (i) this is translated to: if $p$ is regular then
\begin{center}
$a\models p\,|\,AB$ \ iff \ $a\models p\,|\,\cl_{p_A}(B)$ \ \ (for
any $a\in\bar{M}$ and small $B\subset (p|A)(\bar{M})$).
\end{center}

\smallskip (iii) \  Note that in (i) and (ii)  we took
into account only small $B$'s, while Definition \ref{Dreg}
mentions all $B$'s. But this does not matter, since in all of them
the statements with and without `small' are easily seen to be
equivalent.
\end{rmk}

\begin{lem}\label{Lsym}  Suppose that $A$ is small and   $p\in S_1(\bar{M})$ is
$A$-invariant.
\begin{enumerate}
\item[(i)]    $(p(x),x=x)$ is strongly regular \ iff \ $\cl^A_p$
is a closure operator on $\bar{M}$.

\item[(ii)]    Suppose that $(p(x),x=x)$ is strongly regular. Then
$\cl^A_p$ is a pregeometry operator on $\bar{M}$ iff every Morley
sequence in $p$ over $A$ is totally indiscernible.

\item[(iii)]   Suppose that  $p$ is regular.   Then $\cl_{p_A}$ is
a closure operator on $A\cup (p|A)(\bar{M})$; it is a pregeometry
operator  iff every Morley sequence in $p$ over $A$ is totally
indiscernible.
\end{enumerate}
\end{lem}
\begin{proof}
(i) \ Assume that $(p(x),x=x)$ is strongly regular. Then for any
small $B\supseteq A$ we have:
\begin{center} $a\notin\cl_p(B)$ \ iff  \ $a\models p\,|\,B$ \
iff \     $a\models p\,|\,\cl_p(B)$ \  iff \
$a\notin\cl_p(\cl_p(B))$.\end{center}The first and the last
equivalence follow from the definition of $\cl_p$  and the middle
one is by Remark \ref{R22}(i). Thus \ $\cl_p(B)=\cl_p(\cl_p(B)$ \
so $\cl^A_p$ is a closure operator on $\bar{M}$. The other
direction is similar.

\smallskip (ii) \ Suppose that $(p(x),x=x)$ is strongly regular and that
every   Morley sequence in $p$ over $A$ is totally indiscernible.
To show that $\cl^A_p$ is a pregeometry operator, by part (i), it
suffices to verify the exchange property over finite extensions
$B\supset A$. Let $(a_1,...,a_n)\in B^n$ be a maximal Morley
sequence in $p$ over $A$; note that it is finite since it is in
$B\setminus A$, which is finite. Then
$B\subseteq\cl^A_p(a_1,..,a_n)$ (otherwise any element in the
difference would contradict the maximality) and, since $\cl^A_p$
is a closure operator, we get $\cl^A_p(B)\subseteq
\cl^A_p(a_1,..,a_n)$ and thus $\cl^A_p(B)= \cl^A_p(a_1,..,a_n)$.

To verify exchange,  let $a\models p\,|\,B$ (so $a\notin
\cl^{A}_{p}(B)$) and let $b\in \bar{M}$. Note that
$(a_1,...,a_n,a)$ is a Morley sequence   in $p$ over $A$  and
$\cl^A_p(Ba)=\cl^A_p(a_1,...,a_n,a)$. We have:
\begin{center} $b\notin\cl^A_p(Ba)$ \ iff \ $b\notin \cl^A_p(a_1,...,a_n,a)$
\ iff \  $(a_1,...,a_n,a,b)$ is a Morley sequence   \ iff \
$(a_1,...,a_n,b,a)$ is a Morley sequence   \ iff \
($b\notin\cl^A_p(B)$ and $a\notin\cl^A_p(Bb)$).\end{center}
In particular \  $b\notin\cl^A_p(Ba)$ \ implies   \ $a\notin\cl^A_p(Bb)$. This proves exchange.

\smallskip
(iii) \  This is similar to (i) and (ii). Suppose that $p$ is
regular, let $B\subset (p|A)(\bar{M})$ be small and let
$a\models p|\,A$. Then:
\begin{center} $a\notin\cl_{p_A}(B)$ \ iff  \ $a\models p\,|\,AB$ \
iff \     $a\models p\,|\,\cl_{p_A}(B)$ \  iff \
$a\notin\cl_{p_A}(\cl_{p_A}(B))$;\end{center} the first
equivalence and the last equivalence follow  from the definition
of $\cl_{p_A}$ and the middle one is by Remark \ref{R22}(ii). Thus
$\cl_p(B)=\cl_p(\cl_p(B))$. The other clause is proved as in part
(ii).
\end{proof}

\begin{thm}\label{Pr1}
Suppose that  $A$ is small and $p(x)\in S_1(\bar{M})$ is
$A$-invariant and regular.

\begin{enumerate}
  \item[(i)] If $p$ is symmetric,then $\cl_{p_A}$ is a
pregeometry operator on $A\cup (p\,|\,A)(\bar{M})$, In the case where a
$(p(x),x=x)$ is strongly regular, $\cl^A_p$ is a pregeometry
operator on $\bar{M}$.

\item[(ii)] If $p$ is not symmetric, then there exists  a finite
extension $A_0$ of $A$ and a  $A_0$-definable  partial order
$\leq$ such that every Morley sequence in $p$ over $A_0$ is
strictly increasing.
\end{enumerate}
\end{thm}
\begin{proof}
If $p$ is symmetric then  every Morley sequence in $p$ over $A$ is
totally indiscernible and  (i) follows from Lemma \ref{Lsym}.

Now suppose that $p$ is not symmetric. Then,  as mentioned in the
introduction,  there is a finite extension $A_0\supset A$ and a
Morley sequence $(a,b)$ of length $2$ in $p$ over $A_0$ such that
$\tp(a,b/A_0)\neq \tp(b,a/A_0)$. \  Choose
$\phi(x,y)\in\tp(a,b/A_0)$ witnessing the ``asymmetry": \
$\models\phi(x,y)\rightarrow\neg\phi(y,x)$. \ Then \
$\phi(a,x)\wedge\neg\phi(x,a)\in \tp(b/aA_0)\subset p(x))$  so
$\phi(x,a)\notin p(x)$. We claim that
$$(p\,|\,A_0)(t)\cup\{\phi(t,a)\}\cup\{\neg\phi(t,b)\}$$
is inconsistent. Otherwise, there is $d$ realizing $p\,|\,A_0$
such that \ $\models\phi(d,a)\wedge \neg\phi(d,b)$. Then $d$ does
not realize $p\,|\,(A_0,a)$ (witnessed by $\phi(x,a)$)  so, by
regularity, $p\,|\,(A_0,a)\vdash p\,|\,(A_0,a,d)$  and thus
$b\models p\,|\,(A_0,a,d)$. In particular $b\models p\,|\,(A_0,d)$
and since, by invariance, $\phi(d,x)\in p(x)$   we conclude
$\models\phi(d,b)$. A contradiction.

\smallskip By this claim, and compactness,  we find $\theta(t)\in
p\,|\,A_0 $ such that \begin{center} $\models(\forall
t)(\phi(t,a)\wedge\theta(t)\rightarrow \phi(t,b)).$\end{center}Let
\ $x\lqq y$ \ be \ $(\forall
t)(\phi(t,x)\wedge\theta(t)\rightarrow \phi(t,y)).$ \ Clearly,
$\lqq$ defines a quasi order and  $a\lqq b$. Also:
\begin{center}
$\models \phi(a,b)\wedge
\theta(a)\wedge\neg\phi(a,a)$;\end{center} The first conjunct
follows by our choice of $\phi$,    the second from $a\models
p\,|\,A_0$, and the third from the asymmetry of $\phi$. Altogether
they imply $b\nlqq a$. Thus if $x<y$ is $x\lqq y\wedge y\nlqq x$
we have $a<b$.
\end{proof}

The next examples concern issues of whether symmetric regular
types are definable or even generically stable. But we first give
a case where this is true (although it depends formally on Theorem
\ref{Tpr} of the next section).

\begin{cor}\label{C1}
Suppose that $(p(x),x=x)$ is strongly regular and
symmetric. Then $p(x)$ is generically stable.
\end{cor}
\begin{proof}
If $(p(x),x=x)$ is invariant-strongly regular then, by Theorem
\ref{Pr1}(i) $\cl_p$ is a pregeometry operator on $\bar{M}$, and
then $p(x)$ is generically stable by Theorem \ref{Tpr}(ii).
\end{proof}

\begin{exm}\label{E1}  A symmetric,  definable, strongly regular type
which is  not generically stable: \ Let $L=\{U,V,E\}$ where $U,V$
are   unary   and $E$ is a binary predicate. Consider the
bipartite graph $(M,U^M,V^M,E^M)$ where $U^M=\omega$, $V^M$ is the
set of all finite subsets of $\omega$, $M=U^M\cup V^M$,  and
$E^M=\{(u,v)\,:\,u\in U^M, \ v\in V^M, \ \textmd{and} \ u\in v\}$.
Let $A\subset M$ be finite. Then:
\begin{center} If \ $(c_1,...,c_n)\,,(d_1,...,d_n)\in (U^M)^n$  \
have the same quantifier-free type over $A$ then
$\tp(c_1,...,c_n/A)=\tp(d_1,...,d_n/A)$,\end{center}since the
involution of $\omega$ mapping $c_i$'s to $d_i$'s respectively,
and fixing all the other elements of $\omega$ is an
$A$-automorphism of $M$.  Note that this is expressible by a set
of first-order sentences, so is true in the monster.

Further, if $e_1,...,e_n\in U^M$ are distinct and have the same
type over $A$   then, since every permutation of $\omega$ which
permutes  $\{e_1,...,e_n\}$ and fixes all the other elements of
$\omega$  is an $A$-automorphism of $M$, $(e_1,...,e_n)$ is
totally indiscernible over $A$. This is also expressible by a set
of first-order sentences.

Let   $p(x)\in S_1(M)$ be the type of a ``new" element of $U$ which
does not belong to any element of $V^M$. Then, by the above, $p$
is definable,   its global heir $\bar{p}$ is symmetric, and
$(\bar{p}(x),U(x))$ is strongly regular.
\end{exm}

\begin{exm}\label{E2}
A symmetric, strongly regular type  which is not definable: \
Consider the bipartite graph $(M,U^M,V^M,E^M)$ where $U^M=\omega$,
$V^M$ consists of all finite and co-finite subsets of $\omega$,
$M=U^M\cup V^M$, and $E^M$ is  $\in$. \ Let $\bar{M}$ be the
monster and let $p(x)\in S_1(\bar{M})$ be the type of a  new
element of $U^{\bar{M}}$, which  belongs to all co-finite members
of $V^{\bar{M}}$ (and no others). Arguing as in the previous
example  $(p(x),U(x))$ is strongly regular and symmetric. Since
``being a co-finite subset of $U^M$\," is not definable, $p$ is
not definable.
\end{exm}

\begin{defn}\label{D5}
Let $G$ be a definable group in
$\bar{M}$. \ $G$ is called an \emph{regular group} if for
some global type $p(x)\in S_G(\bar{M})$, $(p(x),``x\in G")$ is
strongly regular (in particular invariant over some small set).
\end{defn}

We will see in Example \ref{field} that non symmetric
regular groups, and even fields, exist.

\begin{thm}\label{Tg}
Suppose that $G$ is a group definable over $\emptyset$, which is regular,
witnessed by $p(x)\in S_G(\bar{M})$. Then:

\begin{enumerate}
\item[(i)]   $p(x)$ is both left and right translation invariant
(and in fact invariant under definable bijections).

\item[(ii]    A formula $\phi(x)$ is in $p(x)$ \  iff \ two left
(right) translates of $\phi(x)$ cover $G$ \ iff \ finitely many
left (right) translates of $\phi(x)$ cover $G$. (Hence $p(x)$ is
the unique ``generic type" of $G$.)

\item[(iii)]   $p(x)$ is definable over $\emptyset$.

\item[(iv)]   $G=G^0$ \ (i.e. $G$ is connected).
\end{enumerate}
\end{thm}
\begin{proof}
(i) \ Suppose that $f:G\longrightarrow G$ is a $B$-definable
bijection and $a\models p\,|\,B$. Since $p\,|\,B\vdash
p\,|\,(B,f(a))$ is not possible, by strong regularity, we get
$f(a)\models p\,|\,B$. Thus $p$ is invariant under $f$.

\smallskip
(ii) \ Suppose that $D\subseteq G$ is defined by  $\phi(x)\in
p(x)$ which is over $A$. Let $g\models p\,|\,A$ and we show \
$G=D\cup g\cdot D$\,. If  $b\in G\setminus D$ then $b$ does not
realize $p\,|\,A$ so, by strong regularity, $g\models
p\,|\,(A,b)$. By (i) $g^{-1}\models p\,|\,(A,b)$, thus $g^{-1}\in
D\cdot b^{-1}$ and $b\in g\cdot D$. This proves \ $G=D\cup g\cdot
D$\,.

For the other direction, if finitely many translates of $\psi(x)$
cover $G$  then at least one of them belongs to $p(x)$  and, by
(i), $\psi(x)\in p(x)$.

\smallskip
(iii) and (iv) follow immediately from (ii).
\end{proof}

\begin{Qst} \ Is every regular group commutative?
\end{Qst}

\section{Homogeneous pregeometries}\label{s4}

If  $(M,\cl)$ is a pregeometry then, as usual, we obtain notions
of independence and dimension: for $A,B\subset M$ we say that $A$
is independent over $B$ if $a\notin\cl(A\setminus \{a\}\cup B)$
for all $a\in A$. Given $A$ and $B$, all subsets of $A$ which are
independent over $B$ and maximal such, have the same cardinality,
called $\dim(A/B)$. $(M,\cl)$ is infinite-dimensional if
$\dim(M/\emptyset)\geq\aleph_0$.

\begin{rmk}\label{R32}
(i)  \ If $\bar{c}$ is a tuple of length $n$ then
$\dim(\bar{c}/B)\leq n$ for any $B$.

\smallskip
(ii) \ If $l(\bar{c})=n$, $|A|\geq n+1$ and $A$ is independent
over $B$ then there is $a\in A$ such that
$a\notin\cl(B\cup\bar{c})$.
\end{rmk}

\begin{defn}\label{Dhom} We call an infinite-dimensional  pregeometry $(M,\cl)$
{\em homogeneous} if for any finite $B\subset M$, the set of
all $a\in M$ such that $a\notin \cl(B)$ is the set of realizations
in $M$ of a complete type $p_B(x)$ over $B$.
\end{defn}

Note that  Definition \ref{Dhom} relates in some way the closure
operation to the first-order structure. But  it does not
say anything about automorphisms, and nothing is being claimed
about the homogeneity or strong homogeneity of $M$ as a first-order structure.
In particular we do not want to assume $M$ to be a saturated structure.

\begin{lem}\label{L34}
Suppose $(M,\cl)$ is a  homogeneous pregeometry.

\begin{enumerate}
\item[(i)]   $p_{\cl}(x)=\bigcup\{p_B(x)\,:$ $B$ finite subset of
$M\}$ \ is a complete 1-type over $M$, which we call the generic
type of the pregeometry $(M,\cl)$.

\item[(ii)]   $a\notin\cl(B)$ \ iff \ $a\models p_{\cl}(x)\,|\,B$.
\ In particular \ $\cl=\cl_{p_{\cl}}$.

\item[(iii)]   $I=(a_i\,:\,i<\omega)$ is independent over $B$ \
iff \ $a_i\models p_{\cl}\,|\,(B,a_0,...,a_{i-1})$ for all $i$. \
In particular, if $M=\bar{M}$  and $p_{\cl}$ is $B$-invariant then
$I$ is independent over $B$ iff it is a Morley sequence in
$p_{\cl}$ over $B$.\end{enumerate}
\end{lem}
\begin{proof}
(i) \ Consistency is by compactness: given $A_1,...,A_n$ finite
subsets of $M$ and $B=A_1\cup...\cup A_n$, clearly
$p_{A_1}(x)\cup...\cup p_{A_n}(x)\subseteq p_B(x)$ and the later
is consistent. Completeness is clear.

(ii) and (iii) are easy.
\end{proof}

\begin{lem}\label{L35} Suppose   $(M,\cl)$ is a  homogeneous pregeometry.
Let $(a_i\,:\,i\in\omega)$ be an $\emptyset$-independent subset of
$M$. Then for any $L$-formula $\phi(x,\bar{y})$ with
$l(\bar{y})=n$, and $n$-tuple $\bar{b}$ from $M$:
\begin{center}
$\phi(x,\bar{b})\in p_{\cl}(x)$ \ \ \ iff  \ \ \ $\models
\wedge_{i\in w}\phi(a_i,b)$ \  for some \
$w\subset\{1,...2n\},\,|w|=n+1$.
\end{center}
In particular $p_{\cl}(x)$ is definable.
\end{lem}
\begin{proof}
If $\phi(x,\bar{b})$ is large (namely in $p_{cl}$), then its negation is small, and
thus if $M\models\neg\phi(a,\bar{b})$ then $a\in\cl(\bar{b})$. By
Remark \ref{R32}(ii), at most $n$ many $a_i$'s can satisfy
$\neg\phi(x,\bar{b})$, hence at least $n+1$ among the first $2n+1$
$a_i$'s satisfy $\phi(x,\bar{b})$. Conversely, if at least $n+1$
$a_i$'s satisfy $\phi(x,\bar{b})$ then, again by Remark
\ref{R32}(ii), $\phi(x,\bar{b})$ can not be small, so it is large.
\end{proof}

In the next Proposition we make use of Definition \ref{D4} from
the previous section.

\begin{prop}\label{Ppr}
Suppose $(M,\cl)$ is a  homogeneous pregeometry. Let $p(x)$ be the
generic type and let $\bar{p}(x)$ be its (unique by definability)
global heir.

\begin{enumerate}
\item[(i)]   $(\bar{M},\cl_{\bar{p}})$  \ is a homogeneous
pregeometry and $\cl$ is the restriction of $\cl_{\bar{p}}$ to
$M$.

\item[(ii)]   If     $(a_1,...,a_n)$  (from $\bar{M}$) is
independent over $A$ then:   \begin{center}
 $\tp(b_1,...,b_n/A)=\tp(a_1,...,a_n/A)$ \ \   iff \ \
$(b_1,...,b_n)$ is independent  over $A$.\end{center}

\item[(iii)]    $\bar{p}(x)$ is $\emptyset$-invariant and
generically stable.

\item[(iv)]    $(\bar{p}(x),x=x)$ is strongly regular.
\end{enumerate}
\end{prop}
\begin{proof}
(i) is a reasonably routine exercise, using the fact that $\bar{p}$ is defined
by the same schema which defines $p$, and is left to the reader. But we will briefly point out
why exchange holds: Note first that $\bar{p}^{(2)}(x_{1},x_{2})$ is also definable, over the same parameters as $p$.
Let $\phi(x_{1},x_{2},\bar{b})$ be an arbitrary formula over $\bar M$, with parameters $\bar{b}$ witnessed.
To prove exchange it suffices to see that $\phi(x_{1},x_{2},\bar{b})\in \bar{p}^{(2)}(x_{1},x_{2})$ iff
$\phi(x_{2},x_{1},\bar{b}) \in \bar{p}^{(2)}(x_{1},x_{2})$. By definability of $\bar{p}^{(2)}$, there are
formulas $\delta_{1}(\bar{z})$ and $\delta_{2}(\bar{z})$ over $M$ such that for any  $\bar{b}'$ from $\bar M$,
we have
\begin{center}
$\phi(x_{1},x_{2},\bar{b}')\in \bar{p}^{(2)}(x_{1},x_{2})$ \ iff \
$\models\delta_{1}(\bar{b}')$, \ \  and\end{center}
\begin{center}
$\phi(x_{2},x_{1},\bar{b}') \in \bar{p}^{(2)}(x_{1},x_{2})$ \ iff
\  $\models \delta_{2}(\bar{b}')$.\end{center} As exchange holds
inside $M$, we have $M\models \forall \bar{z}(\delta_{1}(\bar{z})
\leftrightarrow\delta_{2}(\bar{z}))$. Hence this also holds in
${\bar M}$.

\smallskip
(ii)   We prove it by induction on $n$. For $n=1$, by definition,
we have $a_1,b_1\models \bar{p}\,|\,A$. Now assume true for $n$
and prove for $n+1$. Without loss of generality $A=\emptyset$.
Suppose first that    $\tp(b_1,...,b_{n+1})=\tp(a_1,...,a_{n+1})$.
Let $a'$ realize $p\,|\,(a_1,...,a_{n+1},b_1,...,b_{n+1})$. \ So \
$\tp(a_1,...,a_n,a_{n+1})=\tp(a_1,...,a_n,a')$.  On the other
hand, by the induction   assumption (over $\emptyset$),
$(b_1,...,b_n)$ is independent, so independent over $a'$ (by
symmetry). By induction assumption applied over $a'$,
$\tp(b_1,...,b_n,a')=\tp(a_1,...,a_n,a')$. Hence
$\tp(b_1,...,b_n,a')=\tp(b_1,...,b_n,b_{n+1})$. As
$a'\notin\cl_{\bar{p}}(b_1,...,b_n)$, also $b_{n+1}\notin
\cl_{\bar{p}}(b_1,...,b_n)$. Thus $(b_1,...b_n,b_{n+1})$ is
independent.

The converse (if $(b_1,...,b_{n+1})$ is independent then it
realizes $\tp(a_1,...,a_{n+1})$) is proved in a similar fashion
and left to the reader.

\smallskip
(iii)   By part (i), Lemma \ref{L35} also applies to the
pregeometry $\cl_{\bar{p}}$. Let $(a_i\,:\,i\in\omega)$ be
$\cl_{\bar{p}}$-independent. Then $\bar{p}(x)$ is defined over
$(a_i\,:\,i\in\omega)$ as in Lemma \ref{L35}. But if
$(b_i\,:\,i\in\omega)$ has the same type as $(a_i\,:\,i\in\omega)$
then, by (ii), it is also $\cl_{\bar{p}}$\,-independent,  hence
$\bar{p}(x)$ is defined over $(b_i\,:\,i\in\omega)$ in the same
way. This implies that $\bar{p}$ is $\emptyset$-invariant. Thus,
by Lemma \ref{L34}(iii)  a Morley sequence in $\bar{p}$ is the
same thing as an infinite $\cl_{\bar{p}}$\,-independent set. By
Lemma \ref{L35} and Definition \ref{D1}, $\bar{p}$ is generically
stable.

\smallskip
(iv)   By part (i) $(\bar{M},\cl_{\bar{p}})$ is a pregeometry  so,
by Lemma \ref{Lsym}(i), $(\bar{p}(x),x=x)$ is strongly regular.
\end{proof}

We now drop (for a moment) all earlier assumptions and summarize
the situation:

\begin{thm}\label{Tpr}
Let $T$ be an arbitrary theory.

\begin{enumerate}
\item[(i)]   Let $p(x)$ be a global $\emptyset$-invariant  type
such that $(p(x),x=x)$ is strongly regular. Then $(\bar{M},\cl_p)$
is a homogeneous pregeometry.

\item[(ii)]  \ On the other hand, suppose $M\models T$ and
$(M,\cl)$ is a homogeneous pregeometry. Then there is a unique
global $\emptyset$-invariant generically stable type $p(x)$ such
that $(p(x),x=x)$ is strongly regular, and such that the
restriction of $\cl_p$ to $M$ is precisely $\cl$.\end{enumerate}
\end{thm}

\vspace{2mm} \noindent We end this section by pointing out the
connection to exponential fields, as mentioned to us by Kirby. In
\cite{Kirby}, Jonathan Kirby proved that $\ecl(-)$, ``exponential
algebraic closure", as originally defined by Macintyre, gives a
pregeometry on {\em any} exponential field, and this result
extends those of Wilkie \cite{Wilkie} for the complex exponential
field. It is an open question whether for the complex exponential
field, $\ecl(-)$ is {\em homogeneous} in the sense of Definition
\ref{Dhom} above. A positive answer would yield quasiminimality
for the complex exponential field as well as generic stability and
strong regularity of its (unique) exponentially transcendental
type. On the other hand, a positive answer does exist for Zilber's
pseudoexponentiation and some other exponential fields.

\section{Quasiminimal structures}\label{s5}

Recall that a   1-sorted structure $M$  in a countable language is
called quasi-minimal if   $M$ is uncountable and every definable
(with parameters) subset of $M$ is countable or co-countable; the
definition was given by Zilber in \cite{Z1}. Here we investigate
the general model theory of quasiminimality, continuing an earlier
work by Itai, Tsuboi and Wakai \cite{ITW}.

\medskip
Throughout this section fix a quasiminimal structure $M$ and let
 $\bar{M}$ be a saturated elementary extension. The set of all formulas (with
parameters) defining a co-countable subset of $M$ forms a complete
1-type $p(x)\in S_1(M)$; we will call it the generic type of $M$.
If $p(x)$ happens to be definable we will denote its (unique)
global heir by $\bar{p}(x)$.

\begin{rmk}
(i) \ In quasiminimal structures Zilber's countable closure
operator $\ccl$ is defined via $\cl_p$\,:
\begin{center}
$\cl_p^0(A)=\cl_p(A)$, \  $\cl_p^{n+1}(A)=\cl_p(\cl_p^{n}(A))$ \
and \ $\ccl(A)=\bigcup_{n\in\omega} \cl_p^n(A) $.\end{center}

  (ii) \    If $A$ is countable then $\ccl(A)$ is countable, too.
\  $\ccl$ is a closure operator on $M$.

\smallskip
(iii) \  $\cl_p$ is a closure operator \  iff \ $\cl_p=\ccl$
(which is in general  not the case, see Example \ref{E4}).
\end{rmk}

In  Section 2 of  \cite{ITW} Itai, Tsuboi and Wakai, studied the
case when $M$ is strongly $\aleph_1$-homogeneous, in the model
theoretic sense, namely, any partial elementary map between
countable subsets extends to an automorphism of $M$. They proved
in Proposition 2.10 there that if $\ccl$ is   {\em not} a
pregeometry operator then some uncountable subset of $M$ is
totally ordered by a formula, and they also remark that Maesono
has strengthened in one direction the conclusion to: $\Th(M)$ has
the strict order property (Remark 2.16 there). They also proved
(assuming  model-theoretic strong $\aleph_{1}$-homogeneity)that
$\ccl$ is a pregeometry operator whenever $|M|\geq\aleph_2$.

In   Theorem \ref{Tbased} below we will    use  a weaker
assumption than strong $\aleph_1$-homogeneity,  which we call countable
basedness, and derive a stronger  dichotomy similar to the one for
strongly regular types in Theorem \ref{Pr1}.   Then, in  Corollary
\ref{aleph2} we prove that  any  quasiminimal structure of size at
least $\aleph_2$ must be of symmetric type; in particular, $(M,\ccl)$
is a  homogeneous pregeometry (in the sense of the previous section)
after possibly countably many  parameters from $M$
into the language. In Theorem \ref{Tsop} we will prove that the
failure of countable basedness implies the strict order property.

\medskip
Recall from the introduction that $p$ does not split over $A$, if \
$(\phi(x,\bar{b}_1)\leftrightarrow\phi(x,\bar{b}_2))\in p(x)$ \
for all $\phi(x,\bar{y})$ over $A$  and all
$\bar{b}_1,\bar{b}_2\in M$ realizing the same type over $A$.
We talk explicitly about splitting rather than invariance, because $M$ need not be
saturated/homogeneous. Remember that in this section we are taking $p(x)\in S_{1}(M)$ to be
the unique ``generic" type of $M$ consisting of all formulas defining co-countable sets.

\begin{defn}
\begin{enumerate} \item[(i)] $p(x)$  is  {\em based on} $A$ if:
\begin{center}
$p$ does not split over $A$  and $\ccl(AC)\subsetneq M$ for all
finite $C\subset M$\end{center}

\item[(ii)] \ $p(x)$ (or $M$) is  {\em countably based} if there
is a countable $A\subset M$ such that it is based on $A$.
\end{enumerate}
\end{defn}

The technical  condition $\ccl(AC)\subsetneq M$ is satisfied by
any countable  $A$, so countable baseness is equivalent to the
existence of a countable   subset $A$ over which $p$ does not
split. In the proof of Corollary \ref{aleph2} we will use an
uncountable base set, and this is why it  is added: it describes a
relative smallness of $A$ in $M$.

\smallskip
If  $p$ does not split over $A$ then we say that $(a_1,...,a_n)$
is a weak Morley sequence in $p$ over $A$ if $a_{k}$ realizes
$p\,|\,(A,a_1,...,a_{k-1})$ for all relevant $k$. \ As in the case
of global invariant types weak Morley sequences are indiscernible
over $A$. Now we can state one of our main theorems, whose proof
will be presented in a somewhat more general context in the next
section, and then draw an important Corollary.

\begin{thm}\label{Tbased}
Suppose that   $p(x)$ is   based on $A\subset M$. Then $\cl^A_p$
is a closure operator and exactly one of the following two holds:

\begin{enumerate}
\item Every (some) weak Morley sequence in $p$ over $A$ is totally
indiscernible; in this case $(M,\cl^A_p)$ is a homogeneous
pregeometry, $p$ is definable over $A$, $\bar{p}$ (its unique
global heir) is generically stable and $(\bar{p}(x),x=x)$ is
strongly regular.

\item There exists a weak Morley sequence in $p$ over $A$ which is
not totally indiscernible; in this case there is a finite
extension $A_0$ of $A$ and an $A_0$-definable partial order $\leq$
such that every weak Morley sequence in $p$ over $A_0$ is strictly
increasing.
\end{enumerate}
\end{thm}

\begin{cor}\label{aleph2}
If  \,$|M|\geq\aleph_2$\  then $p$ is definable  (hence countably
based)  and the first case of Theorem \ref{Tbased}  holds.
\end{cor}
\begin{proof} Suppose $|M|\geq\aleph_2$ and let $M_0\subset M$ be
$\ccl$-closed of   size $\aleph_1$. First we   show that   $p(x)$
is finitely satisfiable in $M_0$ (actually it is finitely satisfiable in any
uncountable subset of $M$):  for any $\phi(x)\in p(x)$\, $\phi(M)$
is co-countable, so $\phi(M)$ intersects $M_0$. Thus $p$ does not
split over  $M_0$ and,  since $|\ccl(M_0C)|=\aleph_1$ for any
countable $C$,    $p$ is based on $M_0$. Theorem \ref{Tbased}
applies; we will show that Case (1) holds. Otherwise, there is a
definable $<$ and a strictly increasing sequence
$\{a_i\,|\,i<\omega_2\}\subset M$. Then $x<a_{\omega_1}$ and
$a_{\omega_1}<x$ define uncountable mutually disjoint subsets of
$M$. A contradiction. Thus Case (1) holds and the conclusion
follows.
\end{proof}

\begin{thm}\label{Tsop}
If $p$ is not countably based then there   exists a definable
partial order (on some $M^n$)   which has uncountable strictly
increasing chains.
\end{thm}
\begin{proof}
Suppose  that $p$ is not countably based. Then   $p$ is finitely
satisfiable in no countable subset of $M$ and, inductively, we can
find a sequence $(M_i\,|\,i<\omega_1)$ of countable $\ccl$-closed
submodels  and a sequence $(\bar{a}_i\,|\,i<\omega_1)$ of tuples
such that for all $i<\omega_1$:
\begin{enumerate}
\item \ \ $M_i\prec M_{i+1}$ and $M_{\alpha}=\cup_{j<\alpha} M_j$
\ for $\alpha<\omega_1$ a limit ordinal;

\item \ \ $\bar{a}_i\in M_{i+1}$;

\item \ \ $p\,|\,M_i\bar{a}_i$ is not finitely satisfiable in
$M_i$.
\end{enumerate}
For each $i<\omega_1$ witness (3) by an $L$-formula
$\phi_i(x,\bar{y},\bar{z})$ and $\bar{m}_i\in M_i$ such that
$\phi_i(x,\bar{m}_i,\bar{a}_i)\in p(x)$ is not satisfied in $M_i$.
Thus $\neg\phi_i(M,\bar{m}_i,\bar{a}_i)$ is countable and contains
$M_i$. Since $M_{i+1}\supset M_i\bar{a}_i$ is $\ccl$-closed we
have:
\begin{center}
$M_i\subseteq\neg\phi_i(M,\bar{m}_i,\bar{a}_i)\subseteq M_{i+1}$ .
\end{center}
We will now find uncountable
$S\subset\omega_1$, and $\phi(x,\bar{y},\bar{z})\in L$ and $\bar{m}$ from $M$,
such that for all $i\in S$ such that for all $i\in S$
$\phi_i(x,\bar{y},\bar{z}) = \phi(x,\bar{y},\bar{z})$ and $\bar{m}_i = \bar{m}$: Without
loss of generality assume that the universe of $M$ is $\omega_1$. Then
$C=\{\alpha\in\omega_1\,|\,M_{\alpha}=\alpha\}$ is a club subset
of $\omega_1$. By the Pressing Down Lemma    the function \
$\alpha\mapsto\bar{m}_{\alpha}$ \ is constant on a stationary
$S_1\subset C$, so the $\bar{m}_i$'s are the same for all $i\in S_1$
(say $\bar{m}$). Since there are only countably many possibilities
for the $\phi_i$ there is uncountable $S\subset S_1$ such that
$\phi_i$'s are the same for all $i\in S$ (say $\phi$). The family \
$(\neg\phi(M,\bar{m},\bar{a}_i)\,|\,i\in S\,)$ \ is  a strictly increasing
chain of definable subsets of  $M$, yielding the Theorem.
\end{proof}

As we see now the existence of a definable group operation on a quasiminimal structure has
strong (but easy) consequences.

\begin{thm}\label{Tqg}
Suppose that   $M$ is a  quasiminimal group. Then $p(x)$ is
definable over $\emptyset$, both left and right translation
invariant, and $\bar{M}$ is a regular group, in the sense of
Definition \ref{D5}, witnessed by $\bar{p}(x)$, where ${\bar p}$
is the unique global heir of $p$.
\end{thm}
\begin{proof} Let $X\subseteq M$ be definable. First we claim that
$X$ is uncountable \ iff \  $X\cdot X=M$. \ If $X$ is uncountable,
then $X$ is co-countable, as is $X^{-1}$. So for any $a\in M$,
$a\cdot X^{-1}$ is co-countable, so has nonempty intersection with
$X$. If $d\in X\cap a\cdot X^{-1}$ then $a\in X\cdot X$, proving
the claim.

It follows   that  $p(x)$ is definable over $\emptyset$. In
particular   it  is  based on $\emptyset$ and, by Theorem
\ref{Tbased}, $\cl_p$ is a closure operator on $M$. The
definability of $p$ implies that $\cl_{\bar{p}}$ is also a closure
operator and $(\bar{p}(x),x=x)$ is strongly regular by Lemma
\ref{Lsym}(i). The rest follows from Theorem \ref{Tg}.
\end{proof}

Several examples of quasiminimal structures with ``bad" properties
are given in \cite{ITW}. For example  $\omega_{1}\times\mathbb{Q}$
equipped with the lexicographic order is quasiminimal, its generic
type is definable over $\emptyset$, but which is non symmetric in
the sense that Case (2) of Theorem \ref{Tbased} holds. We give
here some other examples, including algebraic ones.

\begin{exm}\label{field} A quasiminimal field with countably based ``generic type",
which is ``asymmetric":  \ In fact, every   strongly minimal
structure $M$ of size $\aleph_1$ can be expanded to such a
quasiminimal structure. \ Let again $I=\omega_1\times \mathbb{Q}$
and let $\vartriangleleft$ be the strict lexicographic order on
$I$. Further, let $B=\{b_i\,|\,i\in I\}$ be a basis of the
strongly minimal structure $M$. For each $a\in M$ let $i\in I$ be
$\vartriangleleft$-maximal for which there are $i_1,...,i_n\in I$
such that $a\in\acl(b_{i_1},...,b_{i_n},b_i)\setminus
\acl(b_{i_1},...,b_{i_n})$;   Clearly, $i=i(a)$ is uniquely
determined. \ Now, expand $(M,..)$ to $(M,<,...)$ where $b< c$ iff
$i(b) \vartriangleleft  i(c)$. \ We will prove that $(M,<,...)$ is
quasiminimal. Suppose that $M_0\prec M$ is a countable,
$<$-initial segment of $M$ and that  $B\setminus M_0$ does not
have $\leq$-minimal elements, and let $a,a'\in M\setminus M_0$.
Then there is an automorphism of $(B,\vartriangleleft)$   fixing
$B\cap M_0$ pointwise   and moving $b_{i(a)}$ to $b_{i(a')}$. It
easily extends to an $M_0$-automorphism of $(M,<,...)$, so \
$b_{i(a)}\equiv b_{i(a')}(M_0)$ (in the expanded structure).\,
Note that replacing $b_{i(a)}$ by $a$ in $B$ (in the definition of
$<$) does not affect $<$, so \ $a\equiv a'(M_0)$ and there is a
single 1-type over $M_0$ realized in $M\setminus M_0$. Since every
countable set is contained in an $M_0$ as above, $(M,<,...)$ is
quasiminimal. The remaining details are left to the reader.
\end{exm}

\begin{Qst} \ Is every quasiminimal field algebraically
closed?\end{Qst}

The following is an example of a quasiminimal structure, which is very far from
being regular: \ $\cl_p(A)\neq \cl_p(\cl_p(A))$ for arbitrarily large countable
$A$'s. In particular the generic type is not countably based.

\begin{exm}\label{E4} A quasiminimal structure where
$\cl_p\neq\ccl$: \ Peretyatkin in \cite{Per} constructed an
$\aleph_0$-categorical theory of a 2-branching tree. Our
quasiminimal structure will a model of this theory. The language
consists of a single binary function symbol $L=\{\inf\}$. Let
$\Sigma$ be the class of  all finite $L$-structures $(A,\inf)$
satisfying:
\begin{enumerate}  \item  \  $(A,\inf)$ is a semilattice;

\item   \  $(A,<)$ is a tree \, (where \, $x< y$ \, iff \,
$\inf(x,y)=x\neq y$);

\item \  (2-branching) \ \ No three distinct, pairwise
$<$-incompatible elements satisfy: \ \
$\inf(x,y)=\inf(x,z)=\inf(y,z)$.\end{enumerate} Then   the
Fraiss\'{e} limit of $\Sigma$ exists and its theory, call it
$T_2$, is $\aleph_0$-categorical and has unique 1-type. If we
extend the language to $\{\inf,<,\perp\}$, where $x\perp y$ stands
for $x\nleq y \wedge y\nleq x$, then $T_2$ has elimination of
quantifiers.

\smallskip Let $(\bar{M},<)$ be  the monster model of $T_2$,   let
$\triangleleft$ be a lexicographic order on $I=\omega_1\times Q$,
and let $C=\{c_i\,|\,i\in \omega_1\times Q\}$ be $<$-increasing.
Then we can find a sequence of countable models $\{M_i\,|\,i\in
\omega_1\times Q\}$ satisfying:
\begin{enumerate}
\item[(a)] \ \ $M_i\prec M_j$ \ for all $i\vartriangleleft j$ ;

\item[(b)] \ \  $M_i\cap C=\{c_j\in C\,|\,j\trianglelefteq i\}$
for all $i$;

\item[(c)] \ \ $M_i\cap C<M_j\setminus M_i$ \ for all
$i\vartriangleleft j$. \end{enumerate}Finally,   let
$M=\bigcup\{M_i\,|\,i\in I\}$. \ Clearly, $C$ is   an  uncountable
branch in $M$. Moreover, by (c), any other branch  is completely
contained in some $M_i$, and is so countable. This suffices to
conclude that $M$ is quasiminimal and that the generic type is
determined by $C<x$.

\smallskip Fix $c_i\in C$ and $a\in M\setminus C$ with $c_i<a$.
Note that  $x\nleq c_i$ is small, so  $M_j\subset\cl_p(c_i)$ for
all $j\vartriangleleft i$. Also,   $x\perp a$ is large so
$\cl_p(a)$ is the union of branches containing $a$.   Since
$c_i\in \cl_p(a)$ we have $M_j\subseteq \cl_p(c_i)\subset
\cl_p^2(a)$;  since $M_j \nsubseteq \cl_p(a)$ we conclude that
$\cl_p(a)\neq\cl_p^2(a)$ and $\cl_p$ is not a closure operator.
Similarly, for any countable $A\subset M$ we can find $a,c_i$ as
above much bigger than $A$,  and thus both realizing $p\,|\,A$.
Then $x\perp a\wedge\neg (x\perp c_i)\in p(x)$ witness that $p(x)$
splits over $A$.
\end{exm}

\section{A general dichotomy theorem}\label{s6}

In this section we outline a dichotomy theorem, which yields
Theorem \ref{Tbased} (the countably based quasiminimal case), and
also subsumes the global strongly regular case (Theorem
\ref{Pr1}). Our set up will also apply to $\kappa$-quasiminimal
structures, for $\kappa$ regular  (defined in the obvious way). We
fix an arbitrary model $N$ (allowing the possibility that $N =
\bar M$), and a complete $1$-type $p(x)$ over $N$. The operator
$\cl_{p}$ defined on subsets of $N$ may not be a closure operator,
but generates one:

\begin{defn}
For $X\subset N$ define: \
$\Cl_p(X)=\bigcup\{\cl_p^n(X)\,|\,n\in\omega\}$ \ where  \
 $\cl_p^0(X)= X$, \
$\cl_p^{n+1}(X)=\cl_p(\cl_p^{n}(X))$.
\end{defn}

Note that $\Cl_p$ is a closure operator and that $\cl_p$ is a
closure operator if and only if $\cl_p=\Cl_p$. If $N$ is
quasiminimal and $p$ is the generic type, then $\Cl_p$ agrees with
$\ccl$. However, in the general case we can easily have
$\Cl_p(\emptyset)=N$ which is not interesting at all; the
interesting case is when $N$ is `infinite dimensional'.

\begin{defn}$A\subseteq N$ is  {\em finitely $\Cl_p$-generated over
$B\subset N$}   if there is a  finite $\bar{a}\subset A$ such that
$A\subseteq\Cl_p(B\bar{a})$; \ if $B=\emptyset$ then we simply say
that $A$ is finitely $\Cl_{p}$-generated.
\end{defn}

The interesting case for us is when  $N$ itself is not finitely
$\Cl_p$-generated, which {\em  will be assumed from now on}. This
is already a weak regularity  assumption on $p$, as we will see in
the next section where it is proved that   $p$ is `locally
strongly regular'. The `relative smallness' of $A\subset N$ is
expressed  by: $N$ is not finitely $\Cl_p$-generated over $A$.

\begin{defn}   A sequence $\{a_i\,|\,i\in\alpha\}$ is
\emph{$\cl_p$-free over $B\subset N$} if  for all $i\leq\alpha$ \
$a_i\notin\cl_p(B\cup\{a_i\,|\,j<i\})$; \ $\cl_p$-free means
$\cl_p$-free over $\emptyset$. \ Similarly   {\em $\Cl_p$-free}
sequences are defined.
\end{defn}

\begin{defn}
$p$ is {\em based on} $A\subset N$ if  $p$ does not split over
$A$, and $N$ is not finitely $\Cl_p$-generated over $A$.
\end{defn}

Note that $\Cl_p$-free sequences are also $\cl_{p}$-free.
Moreover, if $A\subset N$ and $p$ does not split over $A$ then
every $\cl_p$-free sequence  over (any domain containing) $A$ is
indiscernible, by the standard argument.

\begin{lem}\label{Lclosure}
  Suppose that  $p(x)$ is  based on $A$. Then

\begin{enumerate}
\item[(i)]     $\cl^A_p$ is a closure operator on $N$.

\item[(ii)]     $(N,\cl^A_p)$ is a pregeometry iff every
$\cl_p$-free sequence over $A$ is totally indiscernible.
\end{enumerate}
\end{lem}
\begin{proof} Without loss of generality  assume $A=\emptyset$.

\smallskip
(i)    Assuming  that $\cl_p$ is not a closure operator we will
find a non-indiscernible $\cl_p$-free sequence over $C$ for some
finite $C\subset N$, which is in contradiction with non-splitting
over $\emptyset$. So suppose that $\cl_p$ is not a closure
operator. Then there are a finite $C\subset N$ and $a\in N$ such
that $a\in\cl_p^2(C)\setminus\cl_p(C)$. Since $a\notin \cl_p(C)$
we have $a\models p\,|\,C$ so, since $N$ is not finitely
$\Cl_p$-generated, there are     $a_1,a_2\in N$   such that
$(a_1,a_2)$ is a $\Cl_p$-free sequence over $Ca$. We will prove
that $(a,a_1,a_2)$    is not indiscernible over $C$; since it is
$\cl_p$-free over $C$  we have a  contradiction as $p$ does not
split over $C$.

Witness $a\in\cl_p^2(C)$ by a     formula
$\varphi(x,\bar{b})\in\tp(a/C\bar{b})$ which is not in $p(x)$,
where $\varphi(x,y_1,...,y_n)$ is over $C$ and
$(b_1,...,b_n)=\bar{b}\in \cl_p(C)^n$. Choose
$\theta_i(y_i)\in\tp(b_i/CA)$ witnessing $b_i\in\cl_p(C)$ (i.e.
$\theta_i(y_i)\notin p(y_i)$) and let \ \ $x_1\equiv_{\varphi}x_2$
\ denote the formula
\begin{center}
$(\forall \bar{y})\left(\bigwedge_{1\leq i\leq
n}\theta_i(y_i)\rightarrow (\varphi(x_1,\bar{y})\leftrightarrow
\varphi(x_2,\bar{y}))\right)$.\end{center} It is, clearly, over
$C$ and we show $a\nequiv_{\varphi}a_2$: \ from $\models
\varphi(a,\bar{b})$ (witnessing $a\in\cl_p^2(C)$) and
$a_2\notin\Cl_p(C)$ we derive $\models\neg\varphi(a_2,\bar{b})$
and thus $\bar{b}$ witnesses $a\nequiv_{\varphi}a_2$. \ On the
other hand, since all realizations of $\theta_i$'s  are in
$\cl_p(C)$, and since $\tp(a_1/\cl_p(C))=\tp(a_2/\cl_p(C))$, we
have $a_1\equiv_{\varphi}a_2$. Therefore $(a,a_1,a_2)$ is not
indiscernible over $C$.

\smallskip (ii)   Having proved (i), the proof of Lemma
\ref{Lsym}(ii) goes through: \  Let $A_0$ be finite, let
$(a_1,...,a_n)\in A^n_0$ be $\cl_p$-free  over $\emptyset$  such
that $\cl_p(A_0)= \cl_p(a_1,..,a_n)$,    let $a\models p\,|\,A_0$
and let $b\in N$. Then  $(a_1,...,a_n,a)$ is $\cl_p$-free over
$\emptyset$ and $\cl_p(aA_0)=\cl_p(a_1,...,a_n,a)$. We have:
\begin{center} $b\notin\cl_p(aA_0)$ \ iff \ $b\notin \cl_p(a_1,...,a_n,a)$
\ iff \  $(a_1,...,a_n,a,b)$ is   $\cl_p$-free   over $\emptyset$
\ iff \ $(a_1,...,a_n,b,a)$ is   $\cl_p$-free over $\emptyset$  \
iff \ ($b\notin\cl_p(A_0)$ and $a\notin\cl_p(A_0b)$).\end{center}
\end{proof}

\begin{thm}\label{Tgeneral}
Suppose that   $p(x)$ is   based on $A\subset N$. Then $\cl^A_p$
is a closure operator and exactly one of the following two holds:

\begin{enumerate}
\item Every $\cl_p$-free  sequence   over $A$ is totally
indiscernible; \ in this case $(N,\cl^A_p)$ is a homogeneous
pregeometry, $p$ is definable over $A$, $\bar{p}$ (its unique
global heir) is generically stable and $(\bar{p}(x),x=x)$ is
strongly regular.

\item  Otherwise. In which case there is a finite extension $A_0$
of $A$ and an $A_0$-definable partial order $\leq$ such that every
$\cl_p$-free sequence over $A_0$ is strictly increasing.
\end{enumerate}
\end{thm}
\begin{proof}
To simplify the notation assume $A=\emptyset$, it will not affect
the generality. First suppose that every $\cl_p$-free sequence
over $\emptyset$ is symmetric. Then, by Lemma \ref{Lclosure}(i),
$\cl_p$ is a closure operator and, by Lemma \ref{Lclosure}(ii), it
is a pregeometry operator. Since $N$ is not finitely
$\cl_p$-generated it is infinite-dimensional so $(N,\cl_p)$ is a
homogeneous pregeometry and the conclusion follows from
Proposition \ref{Ppr}.

\smallskip
Now suppose otherwise.   Then  over some finite $A_{0}$ there is a
$\cl_p$-free sequence $(a,b)$ over $A_0$ such that $\tp(a,b/A_{0})
\neq \tp(b,a)/A_{0}$. So for some $\phi(x,y)$ over $A_{0}$, we
have:

\begin{enumerate}
\item \ $a\models p\,|\cl_p(A_0)$, \ $b\models p\,|\cl_p(A_0, a)$
\ and \ $\models\phi(a,b)$;

\item \ $\models\phi(x,y)\rightarrow \neg\phi(y,x)$.
\end{enumerate}
We {\em claim}     \ $\phi(N,a)\subsetneq \phi(N,b)$. \  To prove
it, first note that $\models \phi(a,b)\wedge\neg\phi(b,a)$ and
$b\notin\cl_p(Aa)$ imply   $\neg\phi(x,a)\in p(x)$, so \
$\phi(N,a)\subseteq\cl_p(A_0a)$. Now suppose   $d\in\phi(N,a)$ and
we will show $d\in\phi(N,b)$. By the above  $d\in\cl_p(A_0a)$. We
have two possibilities for $d$. The first is $d\in\cl_p(A_0)$,
where $a\equiv b \,(\cl_p(A_0))$ and $\models \phi(d,a)$ imply
$\models \phi(d,b)$ and we're done. The second is
$d\notin\cl_p(A_0)$. Then $a$ and $d$ realize $p\,|\,\cl_p(A_0)$
and, since $p$ does not split over $A_0$, we have \
$(\phi(a,x)\leftrightarrow\phi(d,x))\in p(x)$; \ since
$d\in\cl_p(A_0a)$ we have \
$(\phi(a,x)\leftrightarrow\phi(d,x))\in p\,|\,\cl_p(aA_0)$ and,
since $b\models p\,|\,\cl_p(A_0a)$,   we get $\models
\phi(a,b)\leftrightarrow\phi(d,b)$.  Thus $\models\phi(d,b)$. \
This proves $\phi(N,a)\subseteq \phi(N,b)$.

Finally, the asymmetry of $\phi(x,y)$ implies
$\models\neg\phi(a,a)$ so $a\notin\phi(N,a)$ and
$a\in\phi(N,b)\setminus\phi(N,a)$; this proves that the inclusion
is proper.

\smallskip
 `$\phi(N,x)\subsetneq\phi(N,y)$'
is an $A_0$-definable strict-ordering relation $x<y$ on $N$ and we
have $a<b$. \ By non-splitting, $a'<b'$ is true whenever $(a',b')$
is $\cl_p$-free   over $A_0$, so     every $\cl_p$-free sequence
over $A_0$  is increasing.
\end{proof}

\section{Local regularity}\label{s7}
Here we introduce and study ``locally strongly regular types" and
give applications to  quasiminimal structures  (see Corollary
\ref{C4}).

Given an arbitrary model $M$, $p(x)\in S_{1}(M)$
and $\phi(x)\in p$ we can ask when $p$ extends to a global $M$-invariant $\bar p$,
such that $({\bar p},\phi)$ is strongly regular. When $T$ is stable the situation is well understood,
but we are interested in the general case. Our definition of local strong regularity is actually
a necessary condition for such an extension to exist.
It is convenient to work with types over arbitrary sets (not just models).

\begin{defn}A non-isolated type
$p(x)\in S_1(A)$ is \emph{locally strongly regular via}
$\phi(x)\in p(x)$ if $p(x)$ has a unique extension over $A\bar{b}$
whenever $\bar{b}\in\bar{M}$ is a finite tuple of realizations of
$\phi(x)$ no element of which realizes $p$.
\end{defn}

\begin{prop}\label{P03}
Suppose  that $p(x)\in S_1(M)$ is definable  and  locally strongly
regular via $\phi(x)\in p(x)$, and  let $\bar{p}(x)$ be   its
global heir. Then $(\bar{p}(x),\phi(x))$ is strongly
regular (and of course definable).
\end{prop}
\begin{proof}
Suppose that $(\bar{p}(x),\phi(x))$ is not strongly regular. Then
there are $B=M\bar{b}$ and $a\in\cl_{\bar{p}}(B)\cap\phi(\bar{M})$
and $c\models \bar{p}\,|B$ such that $c$ does not realize
$\bar{p}\,|\,Ba$. Witness $a\in\cl_{\bar{p}}(B)$ by
$\theta(y,\bar{z})$   which is over $M$, implies $\phi(y)$, and
$\models\theta(a,\bar{b})\wedge\neg(d_p\theta)(\bar{b})$  \ (where
$d_p$ is the defining schema of $p$). Similarly, find
$\varphi(x,y,\bar{z})$ over $M$ such that
$\models\varphi(c,a,\bar{b})\wedge\neg(d_p\varphi)(a,\bar{b})$.
\begin{center}
$\models(\exists
y)(\theta(y,\bar{b})\wedge\neg(d_p\theta)(\bar{b})\wedge
\varphi(c,y,\bar{b})\wedge\neg(d_p\varphi)(y,\bar{b}))$.
\end{center}
Since $\tp(c/M\bar{b})$ is an heir of $p(x)$ there is $\bar{m}\in
M$ and $a'$ such that
\begin{center}
$\models \theta(a',\bar{m})\wedge\neg(d_p\theta)(\bar{m}) \wedge
\varphi(c,a',\bar{m})\wedge\neg(d_p\varphi)(a',\bar{m})$.
\end{center}
The first two conjuncts witness   $a'\in\phi(\bar{M})\setminus
p(\bar{M})$ while the last two witness that   $c$ is not a
realization of $\bar{p}\,|\,Ma'$. A contradiction.
\end{proof}

\smallskip
For the sake of this section we will call a sequence
$(a_i\,|\,i\in\alpha)$ a {\em coheir sequence over $C$} if
$\tp(a_j/M(a_i\,|\,i<j))$ is finitely satisfiable in $C$ for any
$j<\alpha$. In particular $\tp(a_{0}/C)$ is finitely satisfiable
in $C$.

\begin{prop}\label{Psymm}
Suppose that $p(x)\in S_1(C)$ is locally strongly regular via
$x=x$  and that there exists an infinite, totally indiscernible
(over $C$) sequence of realizations of $p$ which is a coheir sequence over $C$.
Then $p$  has a global
$C$-invariant extension $\bar{p}$ such that $(\bar{p}(x),x=x)$ is
strongly regular and  generically stable.
\end{prop}
\begin{proof}Let $I=\{a_i\,:\, i\in\omega\}$ be a totally indiscernible (over $C$)
sequence of realizations of  $p$ which is a coheir sequence in $C$ and let
$p_n(x)=\tp(a_{n+1}/Ca_1...a_n)$. We will first prove that each
$p_n(x)$   is locally strongly regular via $x=x$. Suppose, on the
contrary, that   $p_n(x)$ is not locally strongly regular. Then
there are $b_1 ... b_k=\bar{b}$, with none realizing $p_n(x)$,
such that $p_n$ has at least two extensions in
$S_1(C\bar{a}\bar{b})$ (here $\bar{a}=a_1...a_n$). Let
$\varphi(x,\bar{z},\bar{y})$ be over $C$ and such that both
$\varphi(x,\bar{a},\bar{b})$ and $\neg\varphi(x,\bar{a},\bar{b})$
are consistent with $p_n(x)$.

Choose $\theta_i(y_i,\bar{a})\in\tp(b_i/C\bar{a})$ witnessing that
$b_i$ does not realize $p_n(y_i)$ and let \
$\phi(x_1,x_2,\bar{a})$ \ be \ \
\begin{center}
$(\exists \bar{y})\left(\bigwedge_{1\leq i\leq
n}(\theta_i(y_i,\bar{a}) \wedge\neg\theta_i(x_2,\bar{a}))
\wedge\neg(\varphi(x_1,\bar{a},\bar{y})\leftrightarrow
 \varphi(x_2,\bar{a},\bar{y}))\right)$.\end{center} It is,
clearly, over $C$ and we show \
$\models\phi(a_{n+2},a_{n+1},\bar{a})$. \ By our assumptions on
$\varphi$ and $\bar{b}$, there is
$\bar{b}'\equiv\bar{b}(C\bar{a})$ such that $\models
\neg(\varphi(a_{n+1},\bar{a},\bar{b})\leftrightarrow
 \varphi(a_{n+1},\bar{a},\bar{b}'))$. Also $\models
\bigwedge_{1\leq i\leq
n}(\theta_i(b_i,\bar{a})\wedge\theta_i(b_i',\bar{a})
\wedge\neg\theta_i(a_{n+1},\bar{a}))$. \ Thus for any $c\in C$
either $\bar{b}$ or $\bar{b}'$ in place of $\bar{y}$ witnesses
$\models\phi(c,a_{n+1},\bar{a})$ and, since
$\tp(a_{n+2}/C\bar{a}a_{n+1})$ is finitely satisfiable in $C$, we
conclude $\models\phi(a_{n+2},a_{n+1},\bar{a})$.

By total indiscernibility, $\tp(\bar{a}/Ca_{n+1}a_{n+2})$ is finitely satisfiable
in $C$, so there are $\bar{c}\in C$ and $\bar{d}$ such that
\begin{center}
$\bigwedge_{1\leq i\leq n}(\theta_i(d_i,\bar{c})
\wedge\neg\theta_i(a_{n+1},\bar{c}))
\wedge\neg(\varphi(a_{n+2},\bar{c},\bar{d})\leftrightarrow
 \varphi(a_{n+1},\bar{c},\bar{d}))$.\end{center}
$\bigwedge_{1\leq i\leq n}(\theta_i(d_i,\bar{c})
\wedge\neg\theta_i(a_{n+1},\bar{c}))$
witnesses that no $d_i$   realizes $p$, and
$\neg(\varphi(a_{n+2},\bar{c},\bar{d})\leftrightarrow
 \varphi(a_{n+1},\bar{c},\bar{d}))$
witnesses that $p$ does not have a unique extension over
$C\bar{d}$; a contradiction.

\smallskip
Therefore each $p_n(x)$ is locally strongly regular via $x=x$. It
easily follows that $p_I(x)=\cup_{n\in\omega}p_n(x)\in S_1(CI)$ is
locally strongly regular via $x=x$ as well. Moreover, the same is
true    whenever $I'\subset \bar{M}$ is an indiscernible (over
$C$) extension of $I$; then $I'$ is also totally indiscernible
over $C$ and $p_{I'}(x)\in S(CI')$, defined by:
\begin{center}
$\phi(x,\bar{a'})\in p_{I'}$ (where $\phi(x,\bar{y})$ is over $C$
and $\bar{a}'\in I'$) \ iff \ $\phi(x,\bar{a})\in p_{I}(x)$ for
some $\bar{a}\in I$ with $\bar{a}\equiv \bar{a}'\,(C)$
\end{center}
is locally strongly regular  via $x=x$ and does not split over $C$.

\smallskip Now let $J\subset \bar{M}$ be a maximal indiscernible
(over $C$) extension of $I$. By total indiscernibility and
maximality of $J$  no element of $\bar{M}\setminus CJ$ realizes
$p_{J}(x)$, so the local strong regularity implies that $p_{J}$ has
a unique global extension $\bar{p}$. Since $p_{J}(x)$ does not
split over  $C$, it is easy to conclude that  $\bar{p}$ is
$C$-invariant. To show that $(\bar{p}(x),x=x)$ is strongly regular
let  $A \supseteq CI$ be small and we will prove that
$\bar{p}\,|\,A\vdash \bar{p}\,|\,Ab$ for any $b$ which does not
realize $\bar{p}\,|\,A$. \ So fix such a $b$ and let $A_0\subset
A$ be maximal such that $I_0=I\cup A_0$ is a Morley sequence in
$\bar{p}$ over $C$. Then $\bar{p}\,|\,CI_0(x)$ is  $p_{I_0}(x)$
(as defined above for $I'=I_0$), so is locally strongly regular
via $x=x$. The maximality of $A_0$ implies that no $a\in A$
realizes $\bar{p}\,|\,CI_0$  so, by local strong regularity,
$\bar{p}\,|\,CI_0\vdash \bar{p}\,|\,A$. In particular, since $b$
does not realize $\bar{p}\,|\, A$, we have that $b$ does not
realize $\bar{p}\,|\,CI_0$ either. Thus no element of $Ab$
realizes $\bar{p}\,|\,CI_0(x)$ so $\bar{p}\,|\,CI_0(x)\vdash
\bar{p}\,|\,Ab$; in particular $\bar{p}\,|\,A\vdash p\,|\,Ab$ and
$(\bar{p}(x),x=x)$ is strongly regular.

Having proved that  $(\bar{p}(x),x=x)$ is strongly regular  we can
apply Corollary  \ref{C1} to $\bar{M}$ and $\bar{p}$.
Note that $I$ is a Morley sequence in $\bar{p}$ over $C$, and is totally indiscernible, so
$\bar{p}$ is symmetric, and  hence generically stable.
\end{proof}

Our next goal is to prove that the generic type of a quasiminimal
structure is locally  strongly regular via $x=x$. This we will do
in a more general situation, for any  $M$ and   $p\in S_1(M)$ for
which $M$ is not finitely $\Cl_p$-generated  (with notation as in
Section \ref{s6}).

\begin{prop}\label{Pfg}
Suppose    $p\in S_1(M)$   and    $M$ is not finitely
$\Cl_p$-generated. Then:

\begin{enumerate}\item[(i)]   Whenever $I\subseteq M$ is a maximal $\Cl_p$-free
sequence then   $(p\,|\,I)(x)$ is locally strongly regular via
$x=x$, and \ $(p\,|\,I)(x)\vdash p(x)$.

\item[(ii)]    $p(x)$ is locally strongly regular via $x=x$.
\end{enumerate}
\end{prop}
\begin{proof}
(i) \ Let $I\subseteq M$ be a  maximal $\Cl_p$-free sequence. We
will prove that $(p\,|\,I)(x)$ is locally strongly regular via
$x=x$.  After passing to a subset and rearranging $I$ if necessary
we may assume that  $I=\{a_i\,|\,i<\alpha\}$ where $\alpha$ is a
limit ordinal.

Suppose, on the contrary, that there are $d_1,d_2\in \bar{M}$
realizing $p\,|\,I$, a formula $\phi(x,\bar{y})$, and a tuple
$\bar{b}=b_1b_2...b_n\in \bar{M}^n$ such that none of $b_i$'s
realize $p\,|\,I$ and:
\begin{center} $\models \neg\phi(d_1,\bar{b})\wedge
\phi(d_2,\bar{b}))$.\end{center} Choose
$\theta_i(y_i)\in\tp(b_i/I)$ such that $\theta_i(x)\notin p|I(x)$.
To simplify notation    we will assume that $\phi(x,\bar{y})$, as
well as  each $\theta_i(x)$, are over $\emptyset$ (absorbing a few
parameters from $I$ into the language won't hurt the generality).
At least one of
\begin{center}
$\{i<\alpha\,|\,\models \phi(a_i,\bar{b})\}$ \ and \
$\{i<\alpha\,|\,\models \neg\phi(a_i,\bar{b})\}$\end{center} is
cofinal in $\alpha$. Assume the first one is cofinal and let \
$I_0=\{a_i\,|\,\models \phi(a_i,\bar{b})\}$. The cofinality
implies: first  that $p|I$ is finitely satisfiable   in $I_{0}$
(so $\tp(d_1/I)$ is finitely satisfiable in $I_0$); second, that
$p(x)\cup\{\phi(x,\bar{b}\}$ is finitely satisfiable in $I_0$, so there is a type
containing it in  $S_1(Id_1\bar{b})$ which is finitely satisfiable in $I_0$; wlog,
let $d_2$ realizes it. Thus, both $\tp(d_1/I)$ and
$\tp(d_2/Id_1\bar{b})$ are finitely satisfiable in $I_0$. \begin{center} $\models
(\exists\bar{y})(\bigwedge_{1\leq i\leq
n}\theta_i(y_i)\wedge\neg\phi(d_1,\bar{y})\wedge\phi(d_2,\bar{y}))$.
\end{center}
Since $\tp(d_2/Id_1)$ is finitely satisfiable in $I_0$, there is $a_i\in I_0$ such
that:
\begin{center}
$\models (\exists\bar{y})(\bigwedge_{1\leq i\leq
n}\theta_i(y_i)\wedge\neg\phi(d_1,\bar{y})\wedge\phi(a_i,\bar{y}))$.
\end{center}
Since $\tp(d_1/M)$ is finitely satisfiable in $I_0$, there is $a_j\in I_0$ such
that:
\begin{center}
$\models (\exists\bar{y})(\bigwedge_{1\leq i\leq
n}\theta_i(y_i)\wedge\neg\phi(a_j,\bar{y})\wedge\phi(a_i,\bar{y}))$.
\end{center}
Finally, since $a_i,a_j\in M$ there is
$\bar{b}'=b_1'b_2'...b_n'\in M^n$ satisfying:
\begin{center}
$\models \bigwedge_{1\leq i\leq
n}(\theta_i(b_i')\wedge\neg\phi(a_j,\bar{b}')\wedge\phi(a_i,\bar{b}'))$.
\end{center}
But $\bigwedge_{1\leq i\leq n}\theta_i(b_i')$ \ implies \
$\bar{b}'\subset\cl_p(\emptyset)$ and thus  \
$\tp(a_i/\cl_p(\emptyset))\neq \tp(a_j/\cl_p(\emptyset))$. \ A
contradiction. Therefore $(p\,|\,I)(x)$ is locally strongly regular
via $x=x$. \ The maximality of $I$ implies $M=\cl_p(I)$ so, by
local strong regularity of $p\,|\,I$, we have \ $(p\,|\,I)(x)\vdash
p(x)$.

\smallskip
(ii) \ $p\,|\,I$ is locally strongly regular and $(p\,|\,I)(x)\vdash
p(x)$    implies that $p(x)$ is locally strongly regular via
$x=x$.
\end{proof}

\begin{cor}\label{C3} The generic type of a quasiminimal structure
is locally strongly regular via $x=x$.
\end{cor}

\begin{thm}\label{Tccl} Suppose that $p\in S_1(M)$ and  that
$(M,\Cl_p)$ is an infinite dimensional
pregeometry. Then $p$ is definable and $(\bar{p}(x),x=x)$ is
strongly regular and generically stable  (where $\bar{p}$ is the unique global heir of $p$).
\end{thm}
\begin{proof}
First note that, by Proposition \ref{Pfg}(ii), $p$ is locally
strongly regular via $x=x$. Now we will find an infinite, totally
indiscernible sequence of realizations of $p$ which is a coheir sequence over
$M$. Towards this aim we first prove:

\smallskip
{\em Claim.} \ \ If    $A \subset M$ and $a,b\in M$ are
$\Cl_p$-independent over $A$ then    \ $\tp(a,b/A)=\tp(b,a/A)$.

\smallskip It suffices to prove the claim for $A$ finite.
Suppose, on the contrary, that $\phi(x,y)$ is over $A$ and \
$\models \phi(a,b)\wedge\neg\phi(b,a)$. \ Since $b\notin\Cl_p(Aa)$
we have $\phi(a,x)\in p(x)$ and,  since $a\notin\Cl_p(Ab)$, we
have $\neg\phi(b,x)\in p(x)$. \ By infinite dimensionality and
since $Aab$ is finite, there is   $e\in M\setminus\Cl_p(Aab)$.
Then $e$ realizes $p\,|\,Aab$ so $\models
\phi(a,e)\wedge\neg\phi(b,e)$. \ But $a\notin\Cl_p(Ae)$ implies
$\phi(x,e)\in p(x)$ and   $b\notin\Cl_p(Ae)$ implies
$\neg\phi(x,e)\in p(x)$. A contradiction. The claim is proved.

\smallskip

Now, let $I\subseteq M$ be a maximal $\Cl_p$-free sequence. We can
find an infinite $M$-indiscernible sequence $J = (a_{i}:i<\omega)$
such that for all $i$ any formula satisfied by $a_{i}$ over
$M\cup\{a_{j}:j<i\}$ is satisfied by some element of $I$. Using
the claim, we conclude  that $J$ is totally indiscernible over
$M$. Now apply Proposition \ref{Psymm}.
\end{proof}

\begin{cor}\label{C4}
Suppose that $M$ is quasiminimal and let $p(x)\in S_{1}(M)$ be the
``generic type" of $M$ (consisting of formulas defining
uncountable sets). Suppose that $(M,\ccl)$ is a pregeometry. Then
for some countable $\ccl$-free sequence $A\subset M$, $p$ is based
on $A$, and after adding constants for elements of $A$,  is a
homogeneous Then $p$ is countably based and symmetric. Moreover,
as a base set we can choose (some) infinite, countable,
$\ccl$-free sequence.
\end{cor}
\begin{proof}
Clearly, $(M,\ccl)$ is infinite-dimensional so, by Theorem
\ref{Tccl}, the generic type $p$ is definable and symmetric.
Countable baseness and symmetry follow. To prove the `moreover'
part, let $M_0\prec M$ be countable, of infinite $\ccl$-dimension,
such that $p$ is definable over $M_0$ and such that $p$ is the
heir of $p\,|\,M_0$. Further, let $I\subset M_0$ be a maximal
$\ccl$-free subset of $M_0$. Then we can apply Proposition
\ref{Pfg}(i) to $M_0$ and $I$:\ $p\,|\,I\vdash p\,|\,M_0$. In
particular, $p$ is invariant over $I$, so $I$ is a base  for $p$.
\end{proof}

Thus if $M$ is quasiminimal and $(M,\ccl)$ is a pregeometry we
know that the generic type is definable, but    defining schema
may involve parameters. In general parameters are needed   as the
following example shows.

\begin{exm} Quasiminimal structure where $\ccl$ is a pregeometry
operator but the generic type is not $\emptyset$-invariant: \ This
is a slight variation of an example from \cite{ITW}. We have two
unary predicates $U,V$ and binary relation symbols $E$ and $f$.
The domain $M$ is the disjoint union $U(M)\cup V(M)$. $U(M)$ is
uncountable and  $E$ is an equivalence relation on it   having
$\aleph_0$ many classes with all of them but one of size
$\aleph_0$.  $V(M)$ is countable and contains `names' for
$E$-classes, while \  $f:U(M)\longrightarrow V(M)$ \ is an onto
projection (and $f(a)=f(a')$ iff $E(a,a')$). \ Let $p(x)$ be the
generic type: $p(x)$ says that $x$ is in the uncountable class.
Then $\cl_p(\emptyset)=V(M)$, while
$\cl^2_p(\emptyset)=\Cl_p(\emptyset)$ contains also all the
countable $E$-classes. Also $\Cl_p(X)=\Cl_p(\emptyset)\cup X$ so
$(M,\Cl_p)$ is an infinite-dimensional pregeometry. But $p$ is not
$\emptyset$-invariant: all the elements of $U(M)$ have the same
type while $E(x,a)\wedge\neg E(x,b)\in p(x)$ for $a$ in the
uncountable class and $b$ in a countable one. Therefore $p$ is not
definable over $\emptyset$.
\end{exm}

\begin{thm}\label{T9}
Suppose that $G\subseteq M$ is a  definable group and   $p(x)\in
S_G(M)$ is locally strongly regular via $``x\in G"$. Then:

\begin{enumerate}
\item[(i)]   $p(x)$ is both left and right translation invariant
(and in fact invariant under definable bijections).

\item[(ii)]    A formula $\phi(x)$ is in $p(x)$ \  iff \ two left
(right) translates of $\phi(x)$ cover $G$ \ iff \ finitely many
left (right) translates of $\phi(x)$ cover $G$. (Hence $p(x)$ is
the unique generic type of $G$.)

\item[(iii)]   $p(x)$ is definable over $\emptyset$ and $G$ is
connected.

\item[(iv)]    $(\bar{p}(x),``x\in G")$ is strongly regular and
$\bar{G}$ is a definable-regular group. (Here $\bar{p}$ is the
unique heir of $p(x)$ and $\bar{G}\subseteq \bar{M}$ is defined by
$``x\in G"$).
\end{enumerate}
\end{thm}
\begin{proof}
(i) \ Suppose that $f:\bar{G}\longrightarrow \bar{G}$ is an
$M$-definable bijection and $a\models p$. Since $p\vdash
p\,|\,(M,f(a))$ is not possible, by local strong regularity  we
get $f(a)\models p$. Thus $p$ is invariant under $f$.

\smallskip
(ii) \ The local strong regularity of $p(x)$ implies that whenever
$g,g'\in \bar{G}$ do not realize $p$ then $g\cdot g'$ does not
realize $p$ either. It follows that $a\cdot g\models p$ whenever
$a\models p$ and $g\in \bar{G}$ does not realize $p$. Thus:
\begin{center}
$\phi(x)\in p(x)$ \ \ iff \ \ $(\forall y\in
\bar{G})(\neg\phi(y)\rightarrow\phi(y\cdot x))\in p(x)$,
\end{center}
and  \ $\phi(x)\in p(x)$ \  iff \  $\phi(\bar{G})\cup
a^{-1}\cdot\phi(\bar{G})= \bar{G}$.

\smallskip
(iii)  follows immediately from (ii), and then (iv) follows from
Proposition \ref{P03}.
\end{proof}

\bibliographystyle{amsalpha}

\end{document}